\documentclass[a4paper, 12pt]{amsart}
\usepackage{fullpage}
\usepackage{graphicx}
\usepackage{tikz}
\usepackage{amsmath,amsthm, amsfonts, amssymb}
\usepackage{mathrsfs}
\usepackage{enumitem}
\usepackage[english]{babel}
\usepackage [autostyle, english = american]{csquotes}
\MakeOuterQuote{"}
\usepackage{tikz-cd}
\usetikzlibrary{cd}
\usepackage{cite}
\usepackage[hidelinks]{hyperref}

\newcommand{\integers}{\mathbb{Z}}

\newcommand{\complex}{\mathbb{C}}

\newcommand{\vertiii}[1]{{\left\vert\kern-0.25ex\left\vert\kern-0.25ex\left\vert #1 
    \right\vert\kern-0.25ex\right\vert\kern-0.25ex\right\vert}}

\newcommand{\B}{\mathcal{B}}
\newcommand{\C}{\mathcal{C}}
\newcommand{\D}{\mathcal{D}}
\newcommand{\E}{\mathcal{E}}

\newcommand{\J}{\mathcal{J}}
\newcommand{\K}{\mathcal{K}}

\newcommand{\U}{\mathcal{U}}

\newcommand{\X}{\mathcal{X}}

\newcommand{\Z}{\mathcal{Z}}

\newtheorem{thrm}{Theorem}[section]
\newtheorem{lemma}[thrm]{Lemma}
\newtheorem{prop}[thrm]{Proposition}
\newtheorem{cor}[thrm]{Corollary}

\theoremstyle{definition}
\newtheorem{defn}[thrm]{Definition}

\newtheoremstyle{named}{}{}{\itshape}{}{\bfseries}{.}{.5em}{\thmnote{#3's }#1}
\theoremstyle{named}

\begin{document}
\title {Some classes of smooth bimodules over II$_1$ factors  \\ 
and their associated 1-cohomology spaces}
\author{Patrick Hiatt,  Jesse Peterson, Sorin Popa}

\address{Department of Mathematics\\Vanderbilt University, 1326 Stevenson Center, Nashville, TN 37240}
\email{jesse.d.peterson@vanderbilt.edu}
\address{Department of Mathematics, UCLA, Los Angeles CA 90095-1555, USA}
\email{pjhiatt@math.ucla.edu} 
\email{popa@math.ucla.edu} 
\thanks{JP was supported in part by NSF Grant DMS \#1801125 and NSF FRG Grant \#1853989}
\thanks{SP Supported by NSF Grant DMS-1955812  and the Takesaki Endowed Chair at UCLA}
\maketitle

\begin{abstract}
We study several classes of Banach bimodules over a II$_1$ factor $M$, endowed with topologies 
that make them ``smooth'' with respect to $L^p$-norms implemented by the trace $\tau$ on $M$. Thus, letting $M\subset \B= \B(L^2M)$, 
and $2\leq p < \infty$, we consider: $(1)$ the space $\B(p)$, obtained as the completion of $\B$ in the norm $|||T|||_p := \sup \{|\varphi(T)| \mid 
\varphi \in \B^*, \sup\{|\varphi(xYz)| \mid Y\in (\B)_1, x, z \in M\cap (L^pM)_1\} \leq 1 \}$; 
$(2)$ the subspace $\K(p)\subset \B(p)$, obtained as the closure in $\B(p)$ of the space of compact operators $\K(L^2M)$; 
$(3)$ the space $\K_p\subset \B$ of operators that are $||| \ |||_p$-limits of bounded sequences 
of operators in $\K(L^2M)$. We prove that $\K_p$ are  all equal to 
the {\it $\tau$-rank-completion} of $\K(L^2M)$ in $\B$, defined by $\text{\rm q}\K_M:= \{K\in  \B(L^2M) \mid \exists 
K_n \in \K(L^2M), p_n\in \mathcal P(M), \lim_n \|p_n(K-K_n)p_n\|= 0, \lim_n\tau(1-p_n)=0\}$.  
We show that any separable II$_1$ factor $M$ admits non-inner derivations into $\text{\rm q}\K_M$, but that any derivation $\delta:M \rightarrow \text{\rm q}\K_M$ is 
a pointwise limit in $\tau$-rank-topology of  inner derivations. 
\end{abstract}

\section{Introduction}

While Hochschild introduced his cohomology theory for algebras in the mid 1940s, it was around 1970 that 
this theory started to be adapted and systematically 
studied in operator algebras framework  (see the series of papers by Johnson, Kadison and  Ringrose \cite{J72}, \cite{KR71}, \cite{JKR72}). However,  problems related 
to derivations of an operator algebra $M$ with values in special $M$-bimodules, such as $M$ itself, which amounts to the 1-cohomology group of $M$ with coefficients in $M$, 
started to be investigated several years earlier, triggered by Kaplanski's interest in such problems (see e.g. \cite{Ka53} or the footnote in \cite{SW55}). 
In a pioneering result in this direction,  it is shown in \cite{SW55}) that 
any derivation of a commutative Banach algebra must be equal to zero, while in  (\cite{K66}, \cite{Sa66})  is it shown that all derivations of a von Neumann algebra are inner. 
More general $M$-bimodules $\mathcal B$ were soon considered, such as algebras $\B$ that contain $M$, 
notably $M\subset \B=\mathcal B(\mathcal H)$ (see e.g. \cite{Ch80}), or classical ideals in $\mathcal B(\mathcal H)$, 
like the  Schatten-von Neumann $p$-class $c_p(\mathcal H)$, $1\leq p < \infty$ (\cite{JP72}, \cite{J72}, \cite{Ho77}), or the ideal of compact operators (\cite{JP72}). 

Most of the early results in this direction aimed at proving that 
all derivations of an algebra $M$ into an $M$-bimodule $\mathcal B$ are inner, and more generally on showing that all cohomology groups of $M$ with coefficients in $\mathcal B$ vanish, 
H$^n(M, \mathcal B)=0$, $\forall n$. 
But  starting with the work of Johnson in (\cite{J72}, \cite{J74}), an interest towards using the cohomology groups 
H$^n(M, \mathcal B)$ as effective invariants for a von Neumann algebra $M$ has emerged. 
However, while the  amenable-nonamenable  dichotomy could soon be established this way, by showing that a tracial 
von Neumann algebra $M$ is amenable if and only if $H^1(M, \mathcal B)=0$ for any normal dual 
Banach $M$-bimodule $\mathcal B$ (cf \cite{J72}, \cite{C75}, \cite{C76}, \cite{E88}),  by early 1980s all efforts in this direction have stalled. At   
the  ``Operator Algebra Summer School''  in Kingston 1980, where the main directions of research in this area were presented, 
 two cohomology problems were particularly emphasized: $(1)$ whether H$^n(M,M)=0$, $\forall n$, for any II$_1$ factor $M$; 
$(2)$ whether any derivation of a II$_1$ factor $M$ into 
$\mathcal B(\mathcal H)$ is inner when $M\subset \mathcal B(\mathcal H)$ has infinite coupling constant (the case when dim$_M\mathcal H<\infty$ had been settled in the affirmative 
in \cite{Ch80}).  

These problems are still open, but there has been progress on both. On the one hand, problem $(2)$ was shown to be equivalent to the similarity problem, asking whether 
any bounded representation of any C$^*$-algebra $A$, $\pi: A \rightarrow \mathcal B(\mathcal H)$, is similar to a $^*$-representation (i.e., $\exists S\in \mathcal B(\mathcal H)$ invertible 
such that $A \ni x \mapsto S^{-1}\pi(x)S$ is a $^*$-representa- tion), see \cite{Pi01} for several equivalent formulations and a deep analysis of this problem. 
On the other hand, it was shown that H$^n(M, M)=0$, $\forall n$, for many classes of II$_1$ factors with ``good decomposability'' features, such as the property Gamma 
of Murray and von Neumann, existence of Cartan subalgebras, and more generally existence of a ``thin decomposition'' of $M$ with respect to a pair of amenable subalgebras 
(see  \cite{ChPSS03}). But the perception on these problems has changed: one now expects 
that there do exist II$_1$ factors $M$ for which $H^2(M,M)\neq 0$ and $H^1(M, \mathcal B(L^2M \otimes \ell^2\mathbb N) )\neq 0$, 
and that in fact this should be the case for the free group factors $M=L\mathbb F_n , 2\leq n \leq \infty$.  However,  these  
cohomology spaces are expected to be difficult to calculate, and to not be able to make ``fine distinctions'', such as to differentiate  between the free group factors $L\mathbb F_n, 2\leq n \leq \infty$, or show that $L\mathbb F_\infty$ cannot be finitely generated.

A  big impetus towards finding a different cohomology theory for II$_1$ factors, one that would be non-vanishing and calculable,  
providing an efficient invariant that would reflect fine structural properties  of the algebras involved, came in 2001, triggered by Gaboriau's 
successful generalization to orbit equivalence relations $\mathcal R_\Gamma$ arising from 
actions of countable groups by measure preserving transformations $\Gamma \curvearrowright X$ of  Atyiah's and Cheeger-Gromov notion of $L^2$-cohomology of groups, 
leading to his notion of $L^2$-Betti numbers for $\mathcal R_\Gamma$ satisfying $\beta_n^{(2)}(\mathcal R_\Gamma)=\beta^{(2)}_n(\Gamma)$, with the  
striking consequence that free groups of different rank cannot be orbit equivalent (\cite{G01}). 
Since Gaboriau's $L^2$-cohomology for $\mathcal R_\Gamma$ can be viewed as a cohomology theory for the corresponding 
Cartan inclusion $A=L^\infty(X) \subset L(\mathcal R_\Gamma)=M$, of the  group measure space II$_1$ factor associated with the orbit equivalence relation $\mathcal R_\Gamma$, 
it is an invariant for factors $M$ with unique Cartan decomposition (see \cite{P01}), 
for which one can simply define associated $L^2$-Betti numbers as 
$\beta^{(2)}_n(M)=\beta^{(2)}_n(\mathcal R_\Gamma)$.  

But a more interesting ``wishful'' $L^2$-cohomology theory along these lines would be for group factors $M=L\Gamma$  
arising from ICC  groups $\Gamma$, typically without Cartan subalgebras, for which one would like to have an identification between the $L^2$-cohomology of $L\Gamma$ 
and the $L^2$-cohomology of the group $\Gamma$, with the corresponding  
$L^2$-Betti number $\beta^{(2)}_n(L\Gamma)$ coinciding with Atyiah's $L^2$-Betti number of the group, $\beta^{(2)}_n(\Gamma)$. 
This problem was much emphasized by Connes in his talk at MSRI in the Spring of 2001 (\cite{C01}). 

Several attempts were made in this direction: $(a)$   
Connes-Shlyakhtenko proposed in \cite{CS05} an ``everywhere defined'' cohomology of $M$ with coefficients 
in the Murray-von Neumann algebra Aff$(M\otimes M^{op})$ of operators affiliated with $M\otimes M^{op}$; 
$(b)$ Peterson considered in \cite{Pe09} a ``densely defined'' $L^2$-cohomology  theory for II$_1$ factors; $(c)$ Galatan-Popa considered in \cite{GP14} 
a generalized  version of the 1-cohomology with coefficients in $\mathcal K(L^2M)$ in (\cite{JP72}, \cite{P85}),  based on the larger class 
of {\it smooth bimodules}, trying this way to avoid being always equal to $0$, while still vanishing in ``amenable directions''.  

All these attempts have shortcomings: 
\cite{Pe09} encountered the difficulty of having to prove the independence of the cohomology 
on the dense domain of the derivations; \cite{CS05} had to be 
adjusted with some continuity conditions in \cite{Th08}, and that modified version was shown in \cite{PV15}  to always be equal to $0$ 
(this was previously shown in \cite{AK15} and \cite{A14} to hold in certain cases, such as for free group factors); 
of the two  classes of smooth bimodules  proposed in \cite{GP14}, one was shown to produce a cohomology that's always $0$ and the other one 
has not led so far to non-vanishing examples.

Our work in this paper represents a new effort towards identifying a  class of $M$-bimodules  $\mathcal B$ that would allow defining  
a viable cohomology theory, an effective isomorphism invariant, for the II$_1$ factors $M$. To begin with, since our approach is somewhat inspired 
by the $L^2$-cohomology of groups, one expects $\mathcal B$ to depend canonically on $M$ and be related in some ways to the Hilbert space $L^2M$ 
and the space of linear bounded operators acting on it $\mathcal B(L^2M)$. 

Beyond that, a first priority for us was  that the $1$-cohomology with coefficients in $\mathcal B$ should not always vanish, i.e, that there should exist II$_1$ factors 
$M$ that admit non-inner derivations into $\mathcal B$, especially in the case $M=L\Gamma$ with 
$\beta^{(2)}_1(\Gamma)\neq 0$, like $\Gamma=\mathbb F_n$, $2 \leq n \leq \infty$. Another consideration was that $\mathcal B$ should host 
the derivations $\delta_c: L\Gamma \rightarrow \mathcal B$ coming from $1$-cocycles $c:\Gamma \rightarrow \ell^2\Gamma$, which on the 
group algebra $\mathbb C\Gamma = \text{\rm span}\{u_g\}_g$ are of the form $\delta_c(u_g) =[T_f, u_g]$, where $T_f$ is the diagonal operator 
implemented by $f\in \ell^\infty\Gamma$, obtained by ``integrating'' $c$ over the Cayley graph of $\Gamma$ (this is given by the formula  
$f(g)=-c_g(g)$, $\forall g\in \Gamma$). This implicitly means that 
derivations of $M$ into $\mathcal B$ should be uniquely determined by their values on weakly dense $^*$-subalgebras. 
At the same time, one would like $\B$ to have an $M^{op}$-bimodule structure as well, 
commuting with its $M$-bimodule structure, potentially leading to a right $M\overline{\otimes} M^{op}$-module structure on $\mathcal B$.  
One would further hope that whenever $(u_k)^n_{k=1}\subset \mathcal U(M)$ is a finite set 
of unitaries generating $M$ as a von Neumann algebra, the map $\delta \mapsto (\delta(u_k))_k$ 
gives an injective right-$M\overline{\otimes} M^{op}$-modular map from the space of derivations Z$^1(M, \mathcal B)$ into $\mathcal B^n$, 
that would behave well to the quotient by the space of inner derivations B$^1(M, \mathcal B)$, or by its closure  $\overline{\text{\rm B}^1}$ under a suitable topology. 
If such requirements are met,  this would allow associating a first $L^2$-Betti number for $M$, $\beta^{(2)}_1(M)$,  
as the Murray-von Neumann-Lueck dimension of $\tilde{\rm{\text{H}}}^1(M, \mathcal B):=\text{\rm Z}^1(M, \mathcal B)/\overline{\text{\rm B}^1}$ 
viewed as a right $M\overline{\otimes} M^{op}$-module. 

These considerations force  $\mathcal B$ to be somewhat related to $\mathcal K(L^2M)$, the space where $[M_f,u_g]$ takes values. 
So having all this in mind, we consider here the following  spaces.    

For each $p\geq 2$, we consider the  Banach space of ``compact-like operators'' $\mathcal K(p)$ defined as follows. 
We first let $\B^*(p)$ be the space of functionals  
$\varphi$ on $\mathcal B(L^2M)$ with the property that \[
\|\varphi\|_{\B^*(p)}:= \sup\{ |\varphi(xYz)| \mid  Y\in (\B)_1, 
x, y \in  M \cap (L^pM)_1\}<\infty.
\]
We then let $\B(p)$ denote the completion of $\B (L^2M)$ in the norm 
\[
\vertiii{T}_p := \sup \{|\varphi(T)| \mid 
\varphi \in (\B^*(p))_1 \}.
\]
Finally, we denote by $\K(p)\subset \B(p)$ the closure in $\B(p)$ of the space of compact operators $\K(L^2M)$. 
It is immediate to see that $(\K(p), \vertiii{ \, \cdot \, }_p)$ is  both a Banach $M$-bimodule and a Banach $M^{op}$-bimodule. It is also easy to see 
that  for each $X\in \K(p)$ the left-right multiplications by elements in the unit ball of $M$ is $\| \cdot \|_2-\vertiii{ \, \cdot \, }_p$ continuous (smoothness). 

Since the derivations of $M=L\Gamma$ arising from cocycles $c: \Gamma \rightarrow \ell^2\Gamma$ are often implemented by bounded operators 
$M_f\in \B(L^2M)$, we in fact expect that the $M$-bimodules of interest for us consist of bounded operators. We thus also consider the 
spaces $\K_p:=\K(p)\cap \B(L^2M)$, $p\geq 2$. We prove that in fact all $\K_p$, $2\leq p <\infty$, ``collapse'' to just one space, which we show to coincide with the closure in $\B(L^2M)$  of $\K(L^2M)$ in 
the so-called {\it $\tau$-rank metric} $\text{\rm q}_M$, given by its $M$-bimodule structure, $\text{\rm q}_M(S, T)=\inf\{\tau(1-p) + \|p(T-S)p\| \mid p \in \mathcal P(M)\}$.

\begin{thrm}\label{thrm:1.1} For each $p\geq 2$ denote by  $\K_p$ the space of 
operators $T\in \B(L^2M)$ for which there exists a sequence  $K_n \in \K(L^2M)$ such that 
$\sup_n \|K_n\|<\infty$ and $\lim_n \vertiii{T-K_n}_p=0$. Then $\K_p$  coincides with the $\tau$-rank-completion $\text{\rm q}\K_M$ of $\K(L^2M)$ in $\B(L^2M)$.  
\end{thrm}

We note that the $\tau$-rank-completion $\text{\rm q}\K_M$ also coincides with the strong $M$-$M$-comple- tion of $\K(L^2M)$ in the sense of \cite{Ma00}, although we do not take this perspective here.

Any derivation of $M$ into $\text{\rm q}\K_M$ is indeed determined by its values on any weakly-dense $^*$-subalgebra of $M$. In fact, any derivation 
of $M$ into $\text{\rm q}\K_M$ is continuous from the unit ball of $M$ with the $\| \cdot \|_2$-topology to $\text{\rm q}\K_M$ with its $\text{\rm q}_M$-metric. 
Also, $\text{\rm q}\K_M$   
is both a Banach $M$ and $M^{op}$-bimodule and all derivations arising from non-vanishing 1-cocyles $c$ of $\Gamma$ into $\ell^2\Gamma$ described above 
give rise to non-inner derivations of $M=L\Gamma$ into $\text{\rm q}\K_M$. But in fact  any separable II$_1$ factor $M$ (so including the hyperfinite II$_1$ factor) 
admits non-inner derivations into the $M$-bimodule $\text{\rm q}\K_M$:

\begin{thrm}\label{thrm:1.2}
Given any separable diffuse tracial von Neumann algebra $M$, there exist non-inner derivations of $M$ into $\text{\rm q}\K_M$.
\end{thrm}

Thus, while our primary objective of getting a non-vanishing 1-cohomology is indeed being met by the $M$-bimodules $\text{\rm q}\K_M$, the above result 
 shows that the associated  (classic Hochschild) 1-cohomology space H$^1(M, \text{\rm q}\K_M)$, obtained as the 
 quotient of the space of derivations Z$^1(M, \text{\rm q}\K_M)$ by the space of inner derivations $\text{\rm B}^1(M, \text{\rm q}\K_M)$, 
 becomes too ``wild'', certainly  un-calculable. This means one has to take instead 
the quotient of Z$^1(M, \text{\rm q}\K_M)$ by a closure $\overline{\text{\rm B}^1(M, \text{\rm q}\K_M)}$ 
 with respect to some suitable topology on the space of derivations, like one does for the $L^2$-cohomology of groups. 
This should however take into consideration that the  closure of B$^1$ in the $\| \quad \|_2-\text{\rm q}_M$ pointwise convergence on the unit ball of $M$  is too weak for this purpose, as one has the following:

\begin{thrm}\label{thrm:1.3}
 Let  $\delta: M \to \text{\rm q}\K_M$ be a derivation implemented by $T \in \B(L^2 M)$. Then there exists a net of finite-rank operators $K_\iota$ with $\| K_\iota \| \leq \|T \|$ such that 
 $$
 \lim_\iota \text{ \rm q}_M(\delta(x) , [K_\iota , x] )  = 0, \forall x \in M.
 $$ 
 Moreover, if $L^2M$ is separable, then the net can be taken a sequence. 
\end{thrm}

Thus, one has to strengthen the topology on Z$^1(M, \text{\rm q}\K_M)$  
so that the corresponding closure $\overline{\text{\rm B}^1}$ of the space of inner derivations gives all Z$^1(M, \text{\rm q}\K_M)$ in case $M$ is amenable,  
and more generally when $M$ satisfies various ``good decomposition'' 
properties with respect to its amenable subalgebras (like existence of Cartan subalgebras), but is not the entire Z$^1(M, \text{\rm q}\K_M)$ in general, notably for $M=L\mathbb F_n$. 
The resulting relevant 1-cohomology space would then be defined as $\tilde{\text{\rm H}}^1(M, \text{\rm q}\K_M):=\text{\rm Z}^1(M, \text{\rm q}\K_M)/\overline{\text{\rm B}^1}$. 
An alternative, but closely related strategy is to slightly modify the ``target'' $M$-bimodule $\text{\rm q}\K_M$ to a bimodule $\mathcal B$ 
that would still host outer derivations $\delta_c$ arising from non-inner $1$-cocycles of $\Gamma$ 
when $M=L\Gamma$, but would in turn lead to vanishing cohomology when $M$ is amenable, and more generally when $M$ satisfies various ``good decomposition'' 
properties as above. For instance, by taking $\mathcal B$ to be a suitable quotient of $\text{\rm q}\K_M$, or of some modified version of this space. 
We will investigate all these possibilities in a future work.

The paper is organized as follows. In Section 2 we recall some basic definitions about $M$-bimodules and $L^p$-spaces 
associated with a tracial von Neumann algebra $M$. In Section 3 we define for each $p\geq 2$ 
the space $\B^*(p)$ of functionals  
$\varphi$ on $\B:=\B(L^2M)$ with the property that \[
\|\varphi\|_{\B^*(p)}:= \sup\{ |\varphi(xYz)| \leq 1, \forall Y\in (\B)_1, 
x, y \in  M \cap (L^pM)_1\}<\infty.
\]
In Section 4 we consider  its predual, $\B(p)$, obtained as the completion of $\B$ 
in the norm $\vertiii{T}_p := \sup \{|\varphi(T)| \mid 
\varphi \in (\B^*(p))_1 \}$. In Section 5 we define the subspace 
$\K(p)\subset \B(p)$, obtained as the closure in $\B(p)$ of the space of compact operators $\K(L^2M)$, whose dual  identifies naturally 
to the ``normal part'' $\B_{\text{\rm n}}^*(p)$ of $\B^*(p)$. In Section 6 we define the space $\K_p\subset \B$ 
of operators that are $\vertiii{ \, \cdot \, }_p$-limits of bounded sequences of operators in $\K(L^2M)$, define the $\tau$-rank topology on $M$-bimodules, and prove 
Theorem~\ref{thrm:1.1}, showing that all $\K_p$ coincide with the closure $\text{\rm q}\K_M$ of $\K(L^2M)$ in $\B(L^2M)$, in the $\tau$-rank topology (see Theorem~\ref{thrm:6.5}). Then in Section 7 we 
consider the space of derivations of $M$ into $\text{\rm q}\K_M$ and prove Theorems~\ref{thrm:1.2} and \ref{thrm:1.3} (see \ref{thrm:7.3} and \ref{thrm:7.7}).

\section{Preliminaries}
\subsection{Banach bimodules}
 Given a unital Banach algebra $M$ (which will  typically be 
a tracial von Neumann algebra in this paper), a {\it Banach $M$-bimodule} $\B$ is a Banach space 
with {\it left and right  multiplication} operations 
$M \times \B \ni (x, T) \mapsto xT\in  \B$, $\B \times M \ni (T, x) \mapsto Tx \in \B$ (i.e., bilinear maps satisfying  $x(yT)=(xy)T$, 
$(Tx)y=T(xy)$, and $1_MT=T1_M=T$, $\forall x, y\in M, T\in \B$)   that 
satisfy the conditions $\|xT\|_{\B}\leq \|x\|_M \|T\|_{\B}$, $\|Tx\|_{\B}\leq \|T\|_{\B} \|x\|_M$, $\forall x\in M, T\in \B$.

If in addition $\B$ is the dual of a Banach space $\B_*$ and for each $x\in M$ the maps $\B \ni T \mapsto xT \in \B$, 
$\B \ni T \mapsto Tx \in \B$ are continuous with respect to the $\sigma(\B, \B_*)$ topology 
(also called weak$^*$-topology), then 
$\B$ is called a {\it dual  $M$-bimodule}.  Finally, if $M$ is a von Neumann algebra, $\B$ is a dual $M$-bimodule, and for each $T\in \B$ 
the maps $M \ni x \mapsto xT \in \B$, $M\ni T \mapsto Tx \in \B$  are continuous from $(M)_1$ with the $\sigma(M, M_*)$-topology 
to $\B$ with the $\sigma(\B, \B_*)$-topology, then we say that the 
dual $M$-bimodule $\B$ is {\it normal}.

\subsection{Examples}
 A typical example of a Banach $M$-bimodule that we will consider here is when $\B$ is a larger unital Banach algebra that contains $M$ (with $1_M=1_{\B}$), with the left and right 
products $xT, Tx$ for $x\in M$, $T\in \B$, being the restrictions of the product in the larger algebra $\B$. Note that in case $M\subset \B$ 
is an inclusion of von Neumann algebras, $_M\B_M$ is in fact a normal dual $M$-bimodule. 

These examples entail two more classes  of examples of Banach $M$-bimodules. 

On the one hand,  one can take a norm-closed two sided ideal $\J$ in the Banach algebra $\B$, which will have a natural $M$-bimodule structure  
by restriction from $\B$. 

On  the other hand, one can take the dual $\B^*$ of $\B$ with the $M$-bimodule structure given by $\B^*\ni \varphi 
\mapsto x \cdot \varphi \cdot y \in \B^*$, $\forall x, y\in M$, which for $T\in \B$ is defined by $x \cdot \varphi \cdot y (T)=\varphi(yTx)$. This is easily seen to implement a Banach 
dual $M$-bimodule structure on $\B^*$. 

A particular case of this latter type of examples will be of interest to us. Thus, we fix a tracial von Neumann algebra  
$(M,\tau)$ (so $\tau$ is a normal faithful trace state on $M$) and let $M \subset \B=\B(L^2M)$ be its standard representation, where $L^2M$ 
is the Hilbert space obtained by completing $M$ in the norm $\|x\|_2=\tau(x^*x)^{1/2}$, $x\in M$, and $M$ acts on it by left multiplication. 
This makes $\B$  into a dual normal $M$-bimodule. Moreover, since 
$M^{op}$ acts on $L^2M$ as well, by right multiplication,  $\B$ also has dual normal $M^{op}$-bimodule  structure. Since 
$M, M^{op}$ commute  (in fact $M'\cap \B=M^{op}$, $(M^{op})'\cap \B=M$), the two bimodule structures commute, 
in other words they implement a $M \, \otimes_{\rm alg} \, M^{op}$-bimodule structure on $\B$. 

From the preceding remarks, these two bimodules structures on $\B$ entail dual  Banach $M$-bimodule and $M^{op}$-bimodule structures on $\B^*$.

\subsection{Non-commutative $L^p$-spaces}  Recall that $\|y\|_p=\tau(|y|^p)^{1/p}$, $y\in M$, defines a norm on $M$, 
with $\|y\|_p$ being increasing in $p$ and  the limit $\lim_{p \to \infty} \|y\|_p$ 
equal to the operator norm $\| y \|_\infty = \|y\|$. The completion of $M$  
in the norm $\| \cdot \|_p$ is denoted by $L^pM$. One has $L^pM\supset L^{p'}M$ whenever $p'\geq p$. Also, $L^pM$ 
identifies naturally with the  space of densely defined closed operators $Y$ on $L^2M$ that are affiliated with $M$ 
and have the property that $|Y|$ has spectral decomposition $|Y|=\int \lambda \text{\rm d}e_\lambda$ satisfying $\int \lambda^p \text{\rm d}\tau(e_\lambda) < \infty$. 

If $1\leq p\leq \infty$, then  $(L^pM)_1$ is closed in $L^{p'}M$, for any $1\leq p'\leq p$. Moreover, all of the $\| \cdot \|_{p'}$-topologies on the unit ball $(M)_1$ of $M$ for $1\leq p'<\infty$ coincide 
with the $so$-topology on $(M)_1$ and if $p<\infty$, then all $\| \cdot \|_{p'}$-topologies on $(L^pM)_1$, $1\leq p'\leq p$ coincide with the $\| \cdot \|_p$-topology. 

Recall that if $1\leq p <\infty$, then $(L^pM)^*\simeq L^qM$,  where $q=\frac{p}{p-1}$ (with the usual convention 
$1/0=\infty$), the duality being given by $(\xi, \zeta) \mapsto \tau(\zeta^* \xi)$ 
for $\xi \in L^pM$, $\zeta\in L^qM$, viewed as operators affiliated with $M$. This also shows that 
if $y\in M$ and $1\leq p, q \leq \infty$ with $\frac{1}{p}+\frac{1}{q}=1$, then $\|y\|_p=\sup \{|\tau(yz)| \mid z\in (L^qM)_1 \}$. 

Note also that if $x, y \in M$, $\xi \in L^pM$, then $\|x\xi y\|_p\leq \|x\| \|\xi\|_p \|y\|$,  
making $L^pM$ into a Banach $M$-bimodule, which is dual and normal if $1< p < \infty$. Note that if $\xi \in L^pM, \eta\in L^{p'}M$, then $\xi \eta \in L^qM$ where 
$q=\frac{pp'}{p+p'}$.

\subsection{ Smooth bimodules} \label{sec:2.4}

Recall from \cite{GP14} that a Banach $M$-bimodule $\B$ is {\it smooth}, if for any $T\in \B$ the maps $x\mapsto xT$ 
and $x\mapsto Tx$ are continuous from the unit ball of $M$ with its $\| \cdot \|_2$-topology to $\B$ with its 
Banach norm topology. 

A typical example much emphasised in \cite{P85}, \cite{GP14} is when $\B$ is the ideal of compact 
operators $\K(L^2M)\subset \B(L^2M)$, with its $M$-bimodule structure inherited from the $M$-bimodule $\B(L^2M)$. 

Another example, studied in \cite{PR89}, is when $M$  is contained (as a von Neumann subalgbera) in a II$_\infty$ factor $\mathcal M$ with a 
normal semifinite faithful trace $Tr$ and $\B$ is the norm closed $^*$-ideal of ``compact operators'' $\mathcal J(\mathcal M)\subset \mathcal M$, 
consisting of $T\in \mathcal M$ with the property that all spectral projections $e_{[t, \infty)}(T^*T)$ corresponding to $t>0$ have finite trace, $\forall t>0$.

Indeed, in both these cases, it is shown in \cite{P85}, respectively \cite{PR89}, that the Banach $M$-bimodule $\B$, endowed with its corresponding operator norm, 
is smooth in this sense. 

One should mention that in both these examples, the norm $\| \cdot \|$ on the $M$-bimodule $\B$ satisfies a certain {\it operatorial} condition 
(see \cite{GP14}), requiring that if $T\in \B$, then $\|pTp + (1-p)T(1-p)\|=\text{\rm max}\{\|pTp\|, \|(1-p)T(1-p)\|\}$. However, 
in the examples of Banach $M$-bimodules that we will consider in this paper, this property doesn't  hold true in general.

\section{The dual Banach $M$-bimodules $\B^*(p)$, $2\leq p < \infty$  }

We now fix a tracial von Neumann algebra $(M, \tau)$ and we set $\B = \B(L^2M)$. We first consider a one parameter family of $M$ sub-bimodules $\B^*(p) \subset \B^*$, $2\leq p < \infty$,  
defined as spaces of  functionals on $\B$ that are ``$L^p$-smooth 
relative to $M$''.

\begin{defn} Let $2\leq p < \infty$. We denote by $\B^*(p)$ the subspace of functionals $\varphi \in \B^*=\B(L^2M)^*$ with the property that 
\[
\|\varphi\|_{\B^*(p)} := \sup \{|\varphi(xTy)| \text{ }\mid T\in (\B)_1, x, y \in M, \|x\|_p, \|y\|_p \leq 1\}
\]
is finite. Note right away that $\B^*(p)$ 
is a vector subspace of $\B^*$ and that $\| \cdot \|_{\B^*(p)}$ is a norm on it that majorizes the usual norm of functionals in $\B^*$. 
\end{defn}

\begin{prop}\label{prop:3.2}
\begin{enumerate}[label=${\arabic*}^\circ$]
\item
The space $\B^*(p)$ is a Banach space with respect to the norm $\| \cdot \|_{\B^*(p)}$. 
\item
If $2\leq p' \leq  p< \infty$, then $\B^*(p')\subset \B^*(p)$. Moreover, for any $\varphi\in \B^*$ we have $\|\varphi\| \leq \|\varphi\|_{\B^*(p)} \leq \|\varphi\|_{\B^*(p')}$.  
Thus,  $\lim\limits_{p \to \infty} \|\varphi\|_{\B^*(p)} = \inf\limits_{p \to \infty} \|\varphi\|_{\B^*(p)} \geq \|\varphi\|$. 
\end{enumerate} 
\end{prop}
\begin{proof}
$(1)$
It remains to check that $\B^*(p)$ is complete with respect to the $\| \cdot \|_{\B^*(p)}$-norm. So take a Cauchy sequence $(\varphi_n)$ in $(\B^*(p) ,\| \cdot \|_{\B^*(p)}) $. Since the norm $\| \cdot \|_{\B^*(p)}$ majorizes the norm $\| \cdot \|_{\B^*}$, the sequence $(\varphi_n)$ is also Cauchy in $\B^*$.  Let $\varphi$ be its $\| \cdot \|_{\B^*}$-norm limit  in $\B^*$. We claim first that $\varphi \in \B^*(p)$.  Take any $T \in (\B)_1$  and $x,y \in M$ with $\|x\|_p , \|y\|_p \leq 1$. Since $\varphi_n \to \varphi$ with respect to the $\| \cdot \|_{\B^*}$-norm, we can find an $m$ such that 
$$\| \varphi - \varphi_m \|_{\B^*} \leq \| xTy \|^{-1}.$$
In particular, we see that 
\begin{equation}
\begin{split}
|\varphi(xTy)|
&\leq
|\varphi(xTy) - \varphi_m(xTy)| + |\varphi_m(xTy)|
\\
&\leq
\| xTy \| \| \varphi - \varphi_m \| + \sup_{n \geq 1}  |\varphi_n(xTy)|
\\
&\leq
1 + \sup_{n \geq 1}  \|\varphi_n\|_{\B^*(p)}.
\\
\end{split}
\end{equation}
This last quantity is finite since $(\varphi_n)$ was assumed to be Cauchy in $\B^*(p)$. It follows then that $|\varphi(xTy)|$ is uniformly bounded over all $T \in (\B)_1$  and $\|x\|_p , \|y\|_p \leq 1$, and so $\varphi \in \B^*(p)$.

It remains to check that $\varphi_n \to \varphi$ with respect to the $\| \cdot \|_{\B^*(p)}$ norm. To do this, let $S$ be the set of elements $X$ in $\mathcal B$ of the form $xTy$ with $T\in (\mathcal B)_1$ and $x, y\in M$, $\|x\|_p, \|y\|_p\leq 1$. Then we have 
\begin{equation}
\begin{split}
\lim_{n\to \infty}   \|\varphi   - \varphi_n\|_{\B^*(p)}
&=
\lim_{n\to \infty}  
\sup_{X \in S}|\varphi(X)   - \varphi_n(X)|
\\
&=
\lim_{n\to \infty}  
\sup_{X \in S}
\lim_{m\to \infty}
|\varphi_m(X)   - \varphi_n(X)|
\\
&\leq
\lim_{n\to \infty}  
\lim_{m\to \infty}
\sup_{X \in S}
|\varphi_m(X)   - \varphi_n(X)|
\\
&=
\lim_{n,m \to \infty} \|\varphi_m   - \varphi_n\|_{\B^*(p)}.
\\
\end{split}
\end{equation}
Since $(\varphi_n)$ was Cauchy with respect to the $\| \cdot \|_{\B^*(p)}$ norm, it follows that  $(\varphi_n)$ also converge to $\varphi$ with respect to the  $\| \cdot \|_{\B^*(p)}$ norm. This shows $\B^*(p)$ is complete, and thus is a Banach space.

$(2)$
Now suppose $2\leq p' \leq p < \infty$. For any $x \in M$ we have that $\|x\|_{p'} \leq \|x\|_{p}$, so the set $\{xTy \mid  T\in (\B)_1, x, y \in M, \|x\|_p, \|y\|_p \leq 1\}$ is a subset of $\{xTy \mid  T\in (\B)_1, x, y \in M, \|x\|_{p'}, \|y\|_{p'} \leq 1\}$. Taking supremums in the definition of $\| \cdot \|_{\B^*(p)}$, we conclude $\|\varphi\|_{\B^*(p)} \leq \|\varphi\|_{\B^*(p')}$. The rest of the statement follows immediately. 

\end{proof}

\begin{prop}\label{prop:3.3}
\begin{enumerate}
[label=${\arabic*}^\circ$]
\item Let $2\leq p  < \infty$.   If $x, y \in M$ and $\varphi \in \B^*(p)$, then 
$$ \| x\cdot \varphi \cdot y\|_{\B^*(p)} \leq \|x\| \|y\| \|\varphi\|_{\B^*(p)},  \ \| x^{op}\cdot \varphi \cdot y^{op}\|_{\B^*(p)} \leq \|x^{op}\| \|y^{op}\| \|\varphi\|_{\B^*(p)}  .$$  
Thus, the $M$ and $M^{op}$ bimodule structures on $\B^*$ leave $\B^*(p)$ invariant and implement Banach $M$-bimodule 
and $M^{op}$-bimodule structures on $(\B^*(p), \| \cdot \|_{\B^*(p)})$. 

\vskip.05in

\item  The unit ball $(\B^*(p))_1$ is compact in the $\sigma(\B^*, \B)$ topology. 

\vskip.05in

\item  The unit ball $(\B^*(p))_1$ is norm closed in $\B^*$. 

\vskip.05in 

\item For each $x, y \in M$, the map $\B^*(p)\ni \varphi \mapsto x\cdot \varphi \cdot y\in \B^*(p)$ is 
continuous with respect to the $\sigma(\B^*, \B)$-topology. 
\end{enumerate}
\end{prop}
\begin{proof}

$1^\circ$  
Take some $2\leq p <\infty$.  Fix elements $x, y \in M$ and a functional $\varphi \in \B^*(p)$. Let $T \in (\B)_1$ and $x', y' \in M$ with $\|x'\|_p , \|y'\|_p \leq 1$. Then, if we apply $x \cdot \varphi \cdot y $ to $x' Ty'$, we get  
$$
|[x \cdot \varphi \cdot y ](x' T y') |
=
| \varphi( y x' T y' x) |
=
\| x \|   \|y\|
\left|\varphi\left( \frac{y x'}{\|y\| } T \frac{y' x}{\|x\| }
\right) \right|.
$$
Notice that we have the bounds  $\| y x' / \|y\| \|_p \leq 1$ and  $\| y' x / \|x\| \|_p \leq 1$. It follows by definition then that 

$$
|[x \cdot \varphi \cdot y ](x' T y') |
=
\| x \|   \|y\|
\left|\varphi\left( \frac{y x'}{\|y\| } T \frac{y' x}{\|x\| }
\right) \right|
\leq 
\| x \|   \|y\| \| \varphi \|_{\B^*(p)}.
$$
Taking the supremum over all $T \in (\B)_1$ and all  $x', y' \in M$ with $\|x'\|_p , \|y'\|_p \leq 1$ gives the bound $\| x\cdot \varphi \cdot y\|_{\B^*(p)} \leq \|x\| \|y\| \|\varphi\|_{\B^*(p)}$ as desired.

Let us now fix $x^{op}, y^{op} \in M^{op}$. Take $T \in (\B)_1$ and $x', y' \in M$ with $\|x'\|_p , \|y'\|_p \leq 1$. As  we did before, if we apply $x^{op} \cdot \varphi \cdot y^{op} $ to $x' Ty'$, we get 
$$|[x^{op} \cdot \varphi \cdot y^{op} ](x' T y') |
=
| \varphi( y^{op} x' T y' x^{op}) |
=
\| x^{op}\| \| y^{op} \|
\left| \varphi\left(x' \frac{y^{op}}{\| y^{op}\|}  T  \frac{x^{op} }{ \| x^{op}\|} y'\right) 
\right|.
$$
Here this operator  $ y^{op}  T  x^{op} / \| x^{op }\| \| y^{op} \|$ has norm at most 1, so by definition we get 
$$|[x^{op} \cdot \varphi \cdot y^{op} ](x' T y') |
\leq 
\| x^{op}\| \| y^{op} \| \| \varphi \|_{\B^*(p)}.
$$
Taking supremums over all $T, x' , y'$ will give  $\| x^{op}\cdot \varphi \cdot y^{op}\|_{\B^*(p)} \leq \|x^{op}\| \|y^{op}\| \|\varphi\|_{\B^*(p)} $.

$2^\circ$
For $x,y \in M$,  let $S_{x,y} \subset \B^*$ be the set of all functionals $\varphi \in \B^*$ such that $\| x \cdot \varphi \cdot y\|_{\B^*} \leq 1$. It is clear from the  definitions that 
$$(\B^*(p))_1 = \bigcap_{\|x \|_p , \| y \|_p \leq 1} S_{x,y}.$$
Now each of these  sets $S_{x,y}$ is closed in the $\sigma(\B^*, \B)$ topology so $(\B^*(p))_1$ is also closed in this topology. Furthermore, notice since the norm $\| \cdot \|_{\B^*(p)}$ majorizes the operator norm on $\B^*$ that $(\B^*(p))_1 \subset (\B^*)_1$. The Banach-Alaoglu theorem then gives us that  $(\B^*(p))_1$ is compact.

$3^{\circ}$
This is just a consequence of $2^{\circ}$.

$4^{\circ}$
Fix elements $x, y \in M$. From  $1^{\circ}$, we know that the map $\varphi \mapsto x\cdot \varphi \cdot y$ is a well defined linear map from $\B^*(p)$ to itself. It is also a $\sigma(\B^*, \B)$ continuous map on the whole space $\B^*$, so restricting to $\B^*(p)$ proves the claim. 

\end{proof}

\begin{lemma}\label{lemma:3.4} Let $2\leq p <  \infty$ and $\varphi\in \B^*$. 
Assume $\varphi=\omega_{\xi, \eta}$ for some $\xi, \eta \in L^2M$. Then  $\varphi \in \B^*(p)$ if and only if $\xi, \eta \in L^qM$, 
where $q=\frac{2p}{p-2}$, with the 
conventions $\frac{1}{0}=\infty$. Moreover, if this is the case, then $\|\omega_{\xi,\eta}\|_{\B^*(p)}=\|\xi\|_q \|\eta\|_q$. 

\end{lemma}
\begin{proof}
Take elements $x, y \in M$ with $\|x\|_p , \|y\|_p \leq 1$ and an operator $T \in (\B)_1$. By Cauchy-Schwartz, we have a bound
$$
|  \omega_{\xi, \eta}(x T y)  |
=
|\langle  xT y \xi, \eta \rangle  |
=
|\langle  T y \xi, x^*\eta \rangle  |
\leq
\|T y \xi\|_2 \|x^* \eta\|_2
\leq 
\| y \xi\|_2 \|x^* \eta\|_2.
$$
This gives a bound
$$
\|\omega_{\xi, \eta}\|_{\B^*(p)}  
\leq 
\sup_{\|y\|_p\leq 1} \|y \xi\|_2  \sup_{\|x\|_p \leq 1} \|x^* \eta\|_2
=
\sup_{\|y\|_p\leq 1} \|y \xi\|_2  \sup_{\|x\|_p \leq 1} \|x \eta\|_2.
$$
Notice that the reverse inequality also holds. For if $x,y \in M$ are fixed with $\|x\|_p , \|y\|_p \leq 1$, consider the rank-one partial isometry $T_{x,y} \in (\B)_1$ that maps $y \xi$ to $\frac{\|y\xi\|}{\|x^* \eta\|} x^* \eta$. Then 
$$\| y \xi\|_2 \|x^* \eta\|_2 = |\omega_{\xi, \eta}(x T_{x,y} y )|
\leq 
\|\omega_{\xi, \eta}\|_{\B^*(p)}.  
$$
Taking the supremum over $x$ and $y$ gives us the reverse inequality. 

So far we have that 
$$
\|\omega_{\xi, \eta}\|_{\B^*(p)}  
=
\sup_{\|y\|_p\leq 1} \|y \xi\|_2  \sup_{\|x\|_p \leq 1} \|x \eta\|_2.
$$
If we now use the non-commutative version of H\"{o}lder's inequality, then 
$$ \sup_{\|y\|_p\leq 1} \|y \xi\|_2  = \|\xi\|_q,$$ 
where $1/p + 1/q = 1/2$, or $q = \frac{2p}{p-2}$. Similarly, 
$$
\sup_{\|x\|_p \leq 1} \|x \eta\|_2 = \|\eta\|_q.
$$
This gives the desired result $\|\omega_{\xi, \eta}\|_{\B^*(p)}   = \|\xi\|_q \|\eta\|_q$.

\end{proof}

A version of the previous lemma actually works for arbitrary positive finite-rank functionals $\varphi \in \B^*(p)$. 

\begin{lemma}\label{lemma:3.5}
Fix $2\leq p <  \infty$, and let $q = \frac{2p}{p-2}$ be as in the last lemma. Let $\varphi \in B^*(p)$ be of the form $\varphi = \sum_{i=1}^n \omega_{\xi_i, \xi_i}$ where $\xi_1, \xi_2, \dots, \xi_n$ are in $L^{q}M$. Then 
$$\|\varphi\|_{\B^*(p)} = \left|\left| \sum_{i=1}^n \xi_i\xi_i^* \right|\right|_{q/2}.$$
\end{lemma}
\begin{proof}
First we derive a lower bound for $\|\varphi\|_{\B^*(p)}$. Recall that $\|\varphi\|_{\B^*(p)}$ is a supremum over the values $|\varphi(xTy)|$, where $x,y \in M$ are such that  $\|x\|_p, \|y\|_p \leq 1$ and $T \in \B(L^2M)$ is such that $\|T\|\leq 1$. In particular, if we make $T$ the identity on $\B(L^2M)$,
\begin{equation}
\begin{split}
\|\varphi\|_{\B^*(p)} \geq 
\sup_{\|x\|_p,\|y\|_p \leq 1} |\varphi(xy)|
&=
\sup_{\|x\|_p,\|y\|_p \leq 1}
\left| \sum_{i=1}^n \langle xy\xi_i, \xi_i \rangle \right|\\
&=
\sup_{\|x\|_p,\|y\|_p \leq 1}
\left|\sum_{i=1}^n \tau( xy \xi_i \xi_i^* ) \right|\\
&=
\sup_{\|x\|_p,\|y\|_p \leq 1}
 \left| \tau\left(xy \sum_{i=1}^n\xi_i \xi_i^* \right) \right|.\\
\end{split}
\end{equation}
As $x$ and $y$ range over all elements with $p$ norm at most 1, $xy$ can be any element of $M$ with $p/2$ norm at most 1. By density, the above supremum is equal to  
$\sup_{\eta \in (L^{p/2})_1}
 \tau\left(\eta \sum_{i=1}^n\xi_i \xi_i^* \right)$, which by duality is the same as $\left|\left|  \sum_{i=1}^n \xi_i\xi_i^* \right|\right|_{r}$, where $r$ is the H\"{o}lder conjugate of $p/2$. A quick calculation gives
$$ \frac{1}{r} = 1- \frac{1}{p/2} = \frac{p-2}{p} = \frac{1}{q/2}.$$
So we get a lower bound $\left|\left|  \sum_{i=1}^n \xi_i\xi_i^* \right|\right|_{q/2}$ for $\|\varphi\|_{\B^*(p)}$.

Now we prove the reverse inequality. By definition, $\|\varphi\|_{\B^*(p)}$ is the supremum over all sums
$$ \varphi(xTy)
=
 \sum_{i=1}^n \langle xTy\xi_i, \xi_i \rangle
=
\sum_{i=1}^n \langle Ty\xi_i,  x^*\xi_i \rangle,
$$
where $\|x\|_p, \|y\|_p \leq 1$ and $\|T\|\leq 1$. By Cauchy Schwartz, any  one of the inner products $\langle T y \xi_i , x^* \xi_i \rangle$ is bounded by
$$
\langle T y \xi_i , x^* \xi_i \rangle
\leq 
\|x^* \xi_i\|_2 \|T y \xi_i\|_2 
\leq
\|x^* \xi_i\|_2 \|y \xi_i\|_2. 
$$
 Thus, we get the upper bound 

$$\|\varphi\|_{\B^*(p)} 
\leq
\sup_{\|x\|_p,\|y\|_p \leq 1}
\sum_{i=1}^n
\|x^* \xi_i\|_2 \|y \xi_i\|_2
=
\sup_{\|x\|_p,\|y\|_p \leq 1}
\sum_{i=1}^n
\|x \xi_i\|_2 \|y \xi_i\|_2.
$$
If we use H\"{o}lder's inequality, then 
\begin{equation}
\begin{split}
\|\varphi\|_{\B^*(p)} 
&\leq
\sup_{\|x\|_p,\|y\|_p \leq 1}
\left(	\sum_{i=1}^n \|x \xi_i\|_2^2	\right)^{1/2}
\left(	\sum_{i=1}^n\|y \xi_i\|_2^2	\right)^{1/2} \\
&\leq
\sup_{\|x\|_p \leq 1}
\sum_{i=1}^n \|x \xi_i\|_2^2.	\\
\end{split}
\end{equation}
Now, we can write $\|x \xi_i\|_2^2$ in terms of $\tau$ so that 
$$
\|\varphi\|_{\B^*(p)} 
\leq
\sup_{\|x\|_p \leq 1}
\sum_{i=1}^n \tau(x^*x \xi_i\xi_i^*) 
\leq 
\sup_{\|x\|_p \leq 1}
 \tau\left(x^*x \sum_{i=1}^n\xi_i\xi_i^*\right). 
$$
By the same duality argument, this is equal to $\left|\left|  \sum_{i=1}^n \xi_i\xi_i^* \right|\right|_{q/2}$. This completes the proof.
\end{proof}

We remark that  one can calculate an upper bound for $\| \varphi \|_{\B^*(p)}$ for an arbitrary finite-rank functional $\varphi \in \B^*(p)$ by using the polarization identity combined with Lemma~\ref{lemma:3.5}. For a general, not necessarily finite-rank, $\varphi$ one has the following.

\begin{prop} Let $2\leq p <  \infty$ and $\varphi\in \B^*$. 
\vskip.05in

$1^\circ$ If $\varphi\in \B^*(p)$, then $\varphi^*\in \B^*(p)$ and $\|\varphi^*\|_{\B^*(p)}=\|\varphi\|_{\B^*(p)}$. Thus, $\Re \varphi, \Im \varphi \in \B^*(p)$ 
and  $\|\Re \varphi \|_{\B^*(p)}$, $\|\Im \varphi\|_{\B^*(p)}\leq \|\varphi\|_{\B^*(p)}$.

\vskip.05in

$2^\circ$ If $\varphi\in \B^*(p)$,  then its normal and singular parts $($as functionals in $\B^*)$ $\varphi_{\text{\rm n}}, \varphi_{\text{\rm s}}$, belong to $\B^*(p)$, 
with $\|\varphi_{\text{\rm n}}\|_{\B^*(p)}$, $\|\varphi_{\text{\rm s}}\|_{\B^*(p)} \leq \|\varphi\|_{\B^*(p)}$. 
\end{prop}
\begin{proof}
$1^{\circ}$
By the definitions, one obviously have $\|\varphi\|_{\B^*(p)} = \|\varphi^*\|_{\B^*(p)}$ for each $\varphi \in \B^*$. Thus,  $\varphi\in \B^*(p)$ implies $\varphi^* \in \B^*(p)$, 
and hence also the real and imaginary parts of any such $\varphi$, lie in $\B^*(p)$. The given upper bounds then follow from the triangle inequality.

$2^{\circ}$
Let $\varphi$ be any element of $\B^*(p)$, and let $\varphi_{\text{\rm n}}$ and $\varphi_{\text{\rm s}}$ be the normal and singular parts of $\varphi$ respectively. 
Recalling the construction of these functionals, let $p_M$ be the central projection in $\B^{**}$ such that $\varphi_{\text{\rm n}} = p_M \cdot \varphi$ and $\varphi_{\text{\rm s}} = (1-p_M) \cdot \varphi$. If $x$ and $y$ are any elements of $M$ such that $\|x\|_p, \|y\|_p \leq 1$,  then by using the fact that $p_M$ commutes with $M$ we get  
$$
x \cdot \varphi_{\text{\rm n}} \cdot y 
= x \cdot(p_M \cdot \varphi)\cdot y
=p_M( x \cdot \varphi\cdot y).
$$
If we then apply the usual norm from $\B^*$ we have that 
$$
\|x \cdot \varphi_{\text{\rm n}} \cdot y \|
=
\|p_M( x \cdot \varphi\cdot y) \|
\leq 
\|x \cdot \varphi\cdot y\|
\leq 
\|\varphi\|_{\B^*(p)}.
$$
Taking the supremum over all $x$ and $y$ with $\|x\|_p, \|y\|_p \leq 1$, gives us then that $\|\varphi_{\text{\rm n}}\|_{\B^*(p)} \leq \|\varphi\|_{\B^*(p)}$. The same argument with $1 - p_M$ shows that $\|\varphi_{\text{\rm s}}\|_{\B^*(p)} \leq \|\varphi\|_{\B^*(p)}$.

\end{proof}

\begin{cor} Let $2\leq p  < \infty$ and denote $\B^*_{\text{\rm n}}(p)=\{\varphi \in \B^*(p) \mid \varphi=\varphi_{\text{\rm n}}\}$.  
Then $\B^*_{\text{\rm n}}(p)$ is norm closed and $\sigma(\B^*, \B)$-dense in $(\B^*(p), \| \cdot \|_{\B^*(p)})$. 
\end{cor}
\begin{proof}
Take any $2\leq p \leq \infty$. Since the space $(\B^*_{\text{\rm n}})^{*}= \B$, the space is $B^*_{\text{\rm n}}$ is $\sigma(\B^*,\B)$ dense in $\B^*$. Moreover, 
the space $\mathcal L\subset \B^*_{\text{\rm n}}$ obtained as the 
span of functionals of the form $\omega_{\xi, \eta}$ with $\xi, \eta\in \hat{M} \subset L^2M$ is clearly dense in $\B^*_{\text{\rm n}}$ with respect to the usual norm in $\B^*$. 
Since $\mathcal L$ is contained in $\B^*(p)$, this implies that $\B^*_{\text{\rm n}}(p)$ is $\sigma(\B^*,\B)$ dense in $\B^*(p)$. 

Next, consider a Cauchy sequence $\{\varphi_n\}$ in $\B^*_{\text{\rm n}}(p)$. Since $\B^*(p)$ is complete, the sequence converges to some $\varphi \in \B^*(p)$. But for all $p$, the $\| \cdot \|_{\B^*(p)}$ norm dominates the usual norm of functionals in  $\B^*$. Thus, $\varphi$ is the usual norm limit in $\B^*$ of the normal functionals $\varphi_n$, and hence it is normal itself, $\varphi\in\B^*_{\text{\rm n}}$, showing that $\B^*_{\text{\rm n}}(p)$ is norm closed. 

\end{proof}

We end this section by noticing that the norm $\| \cdot\|_{\B^*(p)}$ on the $M$-bimodules $\B^*(p)$ satisfies an interesting property 
with respect to direct sums, which we will however not use in this paper. 

\begin{prop} Let $(M, \tau)$ be a tracial von Neumann algebras and $2\leq p < \infty$. 
Assume $\varphi_1, \varphi_2\in \B^*(p)$ are supported 
by mutually orthogonal projections in $\Z(M)$, i.e., there exist $z_1, z_2\in \mathcal{P}(\Z(M))$ such that $\varphi_i=\varphi_i(z_i \cdot z_i)$, $i=1,2$. 
Then for $p=2$ we have 
$\|\varphi_1 + \varphi_2\|_{\B^*(2)}=\max \{\| \varphi_1\|_{\B^*(2)}, \| \varphi_2\|_{\B^*(2)}\}$ and for $2<p < \infty$ we have 
$\|\varphi_1 + \varphi_2\|_{\B^*(p)}=(\| \varphi_1\|_{\B^*(p)}^q+ \| \varphi_1\|_{\B^*(p)}^q)^{1/q}$, where $q=\frac{p}{p-2}$. 
\end{prop}
\begin{proof}
Let $\varphi = \varphi_1 + \varphi_2$. By definition, $\|\varphi\|_{\B^*(p)}$ is the supremum of 
$|\varphi(xTy)|$ for $x, y \in M$ with $p$-norm at most 1 and $T \in \B(L^2M)$ with norm at most 1. Since $\varphi$ is supported on $z_1 + z_2$, we can restrict the values of $x$ and $y$ we take to only those in $M(z_1+z_2)$, and operators $T$ we take to those supported on $(z_1+z_2)L^2 M$. With this in mind, consider   such  a triple $x$, $y$, and $T$. We can decompose $x=x_1 + x_2$ where $x_1 = z_1 x z_1$ and $x_2 = z_2 x z_2$. Similarly, we can decompose $y = y_1 + y_2$ where $y_1$ and $y_2$ are defined in the same manner. We then define the operator
$$
T = \begin{pmatrix}
T_{11}   &   T_{12}   \\
T_{21}   &   T_{22}   \\
\end{pmatrix},
$$
where here $T_{ij} = z_i T z_j$.
Under this decomposition, we have 
$$|\varphi(xTy)|   = |\varphi_1(x_1 T_{11} y_1) + \varphi_2(x_2 T_{22} y_2)|.
$$

We now wish to maximize this quantity given $\|T\|\leq 1$ and $|x|_p, |y|_p\leq 1$. First, it is clear that it is optimal make the off diagonal terms of $T$ equal to 0, and have the diagonal terms $T_{11}$ and $T_{22}$ have norm 1. Next, we see by properties of the $p$-norm in $M$ that $|x_1|_p^p + |x_2|_p^p = |x|_p^p$ and $|y_1|_p^p + |y_2|_p^p = |y|_p^p$. While varying the $x_i , y_i,  T_{ii}$ under these constraints, we calculate the norm $\|\varphi\|_{\B^*(p)}$ to be the supremum of 
$$\alpha_1 \beta_1 \|\varphi_1\|_{\B^*(p)}
+\alpha_2 \beta_2 \|\varphi_1\|_{\B^*(p)}, 
$$
where the $\alpha_i$ and $\beta_i$ are in $[0,1]$ and satisfy $\alpha_1^p + \alpha_2^p =1 $ and $\beta_1^p + \beta_2^p = 1$. Now let $q$ be such that $1/q + 2/ p =1$, i.e.\ the H\"{o}lder conjugate of $p/2$. Then the discrete version of H\"{o}lder's inequality gives us 

\begin{equation}
\begin{split}
\alpha_1 \beta_1 \|\varphi_1\|_{\B^*(p)}
+\alpha_2 \beta_2 \|\varphi_1\|_{\B^*(p)}  
&\leq 
(\alpha_1^p + \alpha_2^p)^{1/p} (\beta_1^p + \beta_2^p)^{1/p}
( \|\varphi_1\|_{\B^*(p)}^q  +  \|\varphi_2\|_{\B^*(p)}^q )^{1/q} \\
&=
( \|\varphi_1\|_{\B^*(p)}^q  +  \|\varphi_2\|_{\B^*(p)}^q )^{1/q}.
\end{split}
\end{equation}
Moreover, equality is guaranteed to be achieved for some values of $\alpha_i$ and $\beta_i$. This gives us 
$$
\|\varphi\|_{\B^*(p)}
=
( \|\varphi_1\|_{\B^*(p)}^q  +  \|\varphi_2\|_{\B^*(p)}^q )^{1/q}.
$$
Raising both sides to the $q^{th}$ power then completes the proof.

\end{proof}

\section{The Banach bimodules $\B(p)$, $2\leq p < \infty$}

We now consider the natural preduals of the spaces $\B^*(p)$ introduced in the previous section. 

\begin{defn} Let $2\leq p < \infty$. For each $T\in \B=\B(L^2M)$,  denote $\vertiii{T}_p=\sup \{|\varphi(T)| \mid \varphi \in (\B^*(p))_1\}$. 
Noticing that $\vertiii{ \, \cdot \, }_p$ is a norm on $\B$, we denote by $\B(p)$ the completion of $\B$ in this  norm. 
\end{defn}

\begin{lemma}\label{lemma:4.2}  $1^\circ$ For each $T\in \B$, the norms $\vertiii{T}_p$ are increasing in $p$ and majorized by the operator  norm $\|T\|$, 
with $ \lim\limits_{p\to\infty} \vertiii{T}_p = \sup_p \vertiii{T}_p=\|T\|$.  

\vskip.05in

$2^\circ$  If $T\in \B$ and $x, y \in M$, then 
$$\vertiii{x T y}_p \leq \|x\|_p \|T\| \|y\|_p, \ \vertiii{xTy}_p \leq \|x\|  \vertiii{T}_p\|y\|, 
$$
$$
 \vertiii{x^{op}Ty^{op}}_p  \leq \|x^{op}\| \vertiii{T}_p \|y^{op}\|.
$$ 

\end{lemma}
\begin{proof}
$1^{\circ}$
Take $2\leq p \leq p'<\infty$ and  $T \in \B$. Since $(\B^*(p))_1 \subset (\B^*(p'))_1 $ we have that 
\begin{equation}
\begin{split}
\vertiii{T}_p 
&=  
\sup \{ |\varphi(T)| \text{ }\mid \varphi\in (\B^*(p))_1\}
\\
&\leq  
\sup \{ |\varphi(T)| \text{ }\mid \varphi\in (\B^*(p'))_1\}
\\
&=
\vertiii{T}_{p'}.
\end{split}
\end{equation}
So the norms $\vertiii{ \, \cdot \, }_p$ are increasing as $p$ increases. For any finite $2\leq p<\infty$, we also have a bound $\vertiii{T}_p \leq \|T\|$, since the unit ball $(\B^*(p))_1$ is a subset of $(\B^*)_1$. So it follows that $\vertiii{T}_p$ converges as $p$ tends to infinity. 

To find the limit of these norms, take $x$ and $y$ any elements of $M$. Let $\hat{x}$ and $\hat{y}$ be the associated elements of $L^2M$. By Lemma~\ref{lemma:3.4} we have
$$|\langle T \hat{x}, \hat{y} \rangle |
=
|\omega_{\hat{x}, \hat{y}} (T) |
\leq 
\vertiii{T}_p \|\omega_{\xi, \eta}\|_{\B^*(p)}= \vertiii{T}_p \|x\|_q \|y\|_q,
$$
where $q= \frac{2p}{p-2}$. Letting $p$ tend to infinity gives us the bound  
$$|\langle T \hat{x}, \hat{y} \rangle |
\leq 
\|x\|_2 \|y\|_2 \lim_{p \to \infty}   \vertiii{T}_p.
$$
If we take the supremum over all $x, y \in M$  with $\|x\|_2, \|y\|_2 \leq 1$, we get that $\|T\| \leq \lim\limits_{p \to \infty}   \vertiii{T}_p$. The result then follows.

$2^{\circ}$
First consider when $2\leq p < \infty$. Fix $x,y \in M$ and $T\in \B$. If $\varphi$ is an element of $\B^*(p)$, then we have a bound
\begin{equation}
\begin{split}
|\varphi( xTy )|
&=
\|x\|_p \|y\|_p \|T\| \cdot \left|\varphi\left( \frac{x}{\|x\|_p}	\frac{T}{\|T\|}	\frac{y}{\|y\|_p}		\right)\right|
\\
&\leq
\|x\|_p \|y\|_p \|T\| \cdot \|\varphi\|_{\B^*(p)}.
\end{split}
\end{equation}
If we take the supremum over all $\varphi$ with $\|\varphi\|_{\B^*(p)} \leq 1$, this gives the 
$$
\vertiii{xTy }_p  
=
 \sup_{\|\varphi\|_{\B^*(p)} \leq 1 } 
|\varphi( xTy )|
\leq 
\|x\|_p \|T\| \|y\|_p, 
$$
which is the first desired inequality. On the other hand,  one could also note that 
$$
\vertiii{xTy }_p  
=
 \sup_{\|\varphi\|_{\B^*(p)} \leq 1 } 
|\varphi( xTy )|
=
 \sup_{\|\varphi\|_{\B^*(p)} \leq 1 } 
\left|\left(y\cdot \varphi\cdot x\right)( T )\right|.
$$
From Proposition~\ref{prop:3.3}, we know that $\|y\cdot \varphi\cdot x\|_{\B^*(p)} \leq \|x\| \|y\| \|\varphi\|_{\B^*(p)}$. Thus, it follows that 
$$
\vertiii{xTy }_p  
\leq 
 \sup_{\|\varphi\|_{\B^*(p)} \leq \|x\| \|y\| } 
|\varphi( T )|
=
\|x\| \vertiii{T}_p \|y\|.
$$
This gives the second desired inequality.
The case when $x$ and $y$ are elements of $M^{op}$ follows by the exact same reasoning. 
\end{proof}

\begin{prop}
Let $q = \frac{2p}{p-2}$ as before, and let $q' = \frac{2p}{p+2}$ be the H\"{o}lder conjugate of $q$. If $T\in \B(L^2M)$ satisfies $\vertiii{T}_p\leq 1$, then $T$ takes the unit ball of $L^qM$ into 
the unit ball of $L^{q'}M$, thus defining an element $\tilde{T}\in (\B(L^qM, L^{q'}M))_1$. The map $T \mapsto \tilde{T}$ extends 
uniquely to a contractive linear map  from $\B(p)$  into $\B(L^qM , L^{q'} M)$, which is injective when restricted to $\B(L^2M)$.
\end{prop}
\begin{proof}
Noticing that for any $2 \leq p \leq \infty$ one has $q \leq 2 \leq q'$, if $T\in \B(L^2M)$, then for any vector $\xi \in  L^q M \subset L^2 M$ we have 
$$
\| T \xi \|_{q'}
\leq 
\| T \xi \|_{2}
\leq 
\|T\| \|  \xi \|_{2}
\leq 
\|T\| \|  \xi \|_{q}.
$$
Hence, $T$ restricts to a bounded operator $\tilde{T}\in \B(L^qM , L^{q'} M)$.  Moreover, we notice by Lemma~\ref{lemma:3.4}  that if $\xi, \eta$ are vectors in $L^{q}M$, then
$$
|\langle 
T \xi, \eta
\rangle|
=
|\omega_{\xi, \eta} (T)|
\leq
\|\omega_{\xi,\eta}\|_{\B^*(p)} \vertiii{T}_p 
=
\|\xi\|_q \|\eta\|_q \vertiii{T}_p. 
$$
Thus, the bilinear form $u: L^2 M \times L^2 M \to \complex$ given by $u(\xi, \eta) = \langle T \xi, \eta\rangle$ restricts to a  bilinear form on $ L^q M \times L^q M$ with norm at most $\vertiii{T}_p$.  But notice by the noncommutative version of H\"{o}lder's inequality  
$$\sup_{\|\xi\|_q , \| \eta\|_q \leq 1} |\langle 
T \xi, \eta
\rangle|
=
\| T\|_{L^q M \to L^{q'} M}
$$
where here, this norm represents the operator norm in $\B(L^qM , L^{q'} M)$. Thus we conclude that $\| \tilde{T}\|_{L^q M \to L^{q'} M} \leq \vertiii{T}_p$. By the $\vertiii{ \, \cdot \, }_p$-density of $\B(L^2M)$ in $\B(p)$, 
it follows that the map $T \mapsto \tilde{T}$ extends uniquely to a contractive linear map on all $\B(p)$. 

\end{proof}

\begin{prop}\label{prop:4.4}

\vskip.05in 
$1^\circ$ The restriction of the norm $\vertiii{ \, \cdot \, }_p$ to $M\subset \B$ is equal to the norm $\| \cdot \|_{p/2}$ for $L^{\frac{p}{2}}M$. 

\vskip.05in 
$2^\circ$ If $M$ is assumed to be a factor, then the restriction of the norm $\vertiii{ \, \cdot \, }_p$ to $M^{op}\subset \B$ is equal to the operator norm $\|\cdot \|$ on $M^{op}$.

\vskip.05in 
$3^\circ$ If $M^{op}$ is viewed as a subset of $\B(L^qM , L^{q'} M)$,  then the restriction of the norm $\| \cdot \|_{L^qM \to L^{q'}M}$ to $M^{op}$ is equal to the norm $\| \cdot \|_{p/2}$.

\end{prop}

\begin{proof} 

$1^{\circ}$
Fix an element of $x \in M$. Let $x = u |x|$ be the polar decomposition of $x$. Let $1_\B$ be identity operator in $\B$.  Then for any $\varphi \in \B^*(p)$
\begin{equation}
\begin{split}
|\varphi(x)| 
&= 
|\varphi(u |x|^{1/2} 1_{\B} |x|^{1/2}) |
\\
&\leq
\|u |x|^{1/2}\|_p \cdot \| |x|^{1/2}\|_p \cdot \|1_{\B}\| \cdot\|  \varphi\|_{\B^*(p)}
\\
&=
\| |x|^{1/2}\|_p^2 \cdot \|  \varphi\|_{\B^*(p)}
\\
&=
\|x\|_{p/2} \cdot \|  \varphi\|_{\B^*(p)}.
\\
\end{split}
\end{equation}
 Taking the supremum over all $\varphi$ in $(\B^*(p))_1)$ gives the inequality $\vertiii{x}_p \leq \|x\|_{p/2}$.
 
Now we prove the reverse inequality. Let $q = \frac{2p}{p-2}$, as  in Lemma~\ref{lemma:3.4}. Then note that if $\xi$ and $\eta$ are vectors in $L^2M$ such that $\|\xi\|_q =\|\eta\|_q = 1$, Lemma~\ref{lemma:3.4} implies that $|\langle x \xi, \eta \rangle| \leq \vertiii{x}_p$. Thus, we have that

\begin{equation}
\begin{split}
\vertiii{x}_p \geq 
\sup_{\|\xi\|_q =\|\eta\|_q = 1} 
|\langle x \xi, \eta \rangle|.
\end{split}
\end{equation}
If we choose $q'$  such that $\frac{1}{q} + \frac{1}{q'} =1$ and choose $r$ such that $\frac{1}{r} + \frac{1}{q} =\frac{1}{q'}$, then by duality we have
$$
\sup_{\|\xi\|_q =\|\eta\|_q = 1} 
|\langle x \xi, \eta \rangle|
=
\sup_{\|\xi\|_q =1} 
\| x \xi \|_{q'}
=
\|x\|_r,
$$
so we have a bound $\|x\|_r \leq \vertiii{x}_p$. Now by our chosen definition of $r$ we check that 
$$
\frac{1}{r}
=
\frac{1}{q'} - \frac{1}{q}
=
1-\frac{2}{q}
=
\frac{2}{p}.
$$
So indeed, we have $r = p/2$, and the reverse inequality $\|x\|_{p/2} \leq \vertiii{x}_p$ holds. This completes the proof.

$2^{\circ}$
Assume that $M$ is a factor, and take an element $x^{op} \in M^{op}$. By Lemma~\ref{lemma:4.2}, we already know that $\vertiii{x^{op} }_p \leq \| x^{op}\|$, so it suffices to check that $\vertiii{x^{op} }_p \geq \| x^{op}\|$. To do this, we will construct a family of functionals $\varphi \in \B^*(p)$ such that $|\varphi(x^{op})|/\|\varphi\|_{\B^*(p)}$ can come arbitrarily close to $\|x^{op}\|$. From here the result will follow since $\vertiii{x^{op}}_p \geq |\varphi(x^{op})|/\|\varphi\|_{\B^*(p)}$ for all $\varphi \in \B^*(p)$

With this in mind, let's say we choose a self adjoint element $m  \in M$ and a finite list of unitaries $u_1, u_2, \dots, u_n \in M$. Then we can define a linear functional $\varphi \in \B^*(p)$ by 
$$\varphi(T ) 
=
\frac{1}{n}\sum_{i=1}^n 
\langle T(u_im ), u_im \rangle.
$$
By Lemma~\ref{lemma:3.5}, we know that 
$$\|\varphi\|_{\B^*(p)} = 
\left|\left|
\frac{1}{n}\sum_{i=1}^n 
(u_i m) (u_im)^* 
\right|\right|_{q/2}
 = 
\left|\left|
 \frac{1}{n} 
\sum_{i=1}^n 
u_i m^2 u_i^* 
\right|\right|_{q/2}.
$$
If we apply this $\varphi$ to $x^{op}$, we get
$$
\varphi(x^{op})
=
\frac{1}{n}\sum_{i=1}^n 
\langle x^{op}(u_im), u_im \rangle
=
\frac{1}{n}\sum_{i=1}^n 
\langle u_imx, u_im \rangle.
$$
Using that this inner product comes from the trace $\tau$, we can simplify this to be 
$$
\varphi(x^{op})
=
\frac{1}{n}\sum_{i=1}^n
\tau(
 u_imx (u_im)^* )
 =
\frac{1}{n}\sum_{i=1}^n
\tau(
 x m^2 )
  =
\langle
 x ,m^2 
 \rangle.
$$

Now, using that $\vertiii{x}_p \geq |\varphi(x^{op})|/\|\varphi\|_{\B^*(p)}$, we get the following lower bound
$$
\vertiii{x^{op}}_p
\geq 
|\langle x, m^2 \rangle |
\left|\left|
 \frac{1}{n} 
\sum_{i=1}^n 
u_i m^2 u_i ^*
\right|\right|_{q/2}^{-1}.
$$
Note that this holds for any self adjoint $m \in M$ and any choice of unitaries $u_1, u_2, \cdots , u_n$. But by Diximier's averaging property, we know the $\| \cdot \|$-norm closure of the convex hull of the set $\{u m^2 u^* : u \in \U(M)\}$ intersects the center of $M$. In this case, since $M$ was assumed to be a factor, the center of $M$ is trivial. In particular, since the trace of any element in $\{u m^2 u^* : u \in \U(M)\}$ is $\tau(m^2)$, it follows that $\tau(m^2)$ is in the $\| \cdot \|$-norm closure of the convex hull of this set. Now the operator norm $\|\cdot \|$ majorizes the norm $\|\cdot\|_{q/2}$ norm, so the same averaging result is true in the space $L^{q/2} M$. It follows then from the above lower bound that 
$$
\vertiii{x^{op}}_p
\geq 
|\langle x, m^2 \rangle |
\frac{1}{\tau(m^2)}
=
|\langle x, \frac{m^2}{\tau(m^2)} \rangle |,
$$
where $m^2$ can be an arbitrary positive element of $M$. Thus, we conclude that 
$$
\vertiii{x^{op}}_p
\geq 
\sup_{\substack{ m \geq 0,\text{ } \|m\|_1 \leq 1}} |\langle x , m\rangle|,
$$
where here this supremum runs over all positive $m \in M$ with $\|m\|_1 = \tau(m) \leq 1$.

Now if $x$ was assumed to be positive, duality would force this supremum to be $\|x\|$, which would give us the desired reverse inequality  $\|x^{op}\| \leq \vertiii{x^{op}}_p$. In general, we can write $x^{op} = u^{op} |x^{op}| $ to be the polar decomposition of $x^{op}$. Note by part $1^{\circ}$ of Proposition~\ref{prop:3.3}, that the map $T \mapsto u^{op} T$ is an isometry on $\B(p)$ with respect to the $\vertiii{\cdot}_p$-norm. In particular, it follows from the positive case that  
$$
\vertiii{x^{op}}_p
=
\vertiii{\, |x^{op}| \, }_p
\geq
\| \, |x^{op}| \, \|
=
\|x^{op}\|,
$$
so the reverse inequality  holds for a general $x^{op}$, which completes the proof of the first claim.

$3^\circ$ 
We see by H\"{o}lder's inequality that $\|x^{op}\|_{L^qM \to L^{q'}M}$ is equal to $\|x^{op}\|_r$, where $r$ is the solution to the equation $q^{-1} + r^{-1} = (q')^{-1}$. Using the definition of $q$ and $q'$, one gets $r = p/2$, as desired. 

\end{proof}

\begin{cor} If $M$ is a $\text{\rm II}_1$ factor, then the map $\B(p) \ni T \mapsto \tilde{T}\in \B(L^qM, L^{q'}M)$ is not a homeomorphism of Banach spaces. 

\end{cor} 
\begin{proof} Let $x_n \in (M)_1$ be so that $\|x_n\|=1$ but $\|x_n\|_{p/2} \rightarrow 0$. If we take $T_n=x_n^{op}$, then by Proposition~\ref{prop:4.4}.2$^\circ$ 
we have $\vertiii{T_n}_p = \|T_n\|=\|x_n\|=1$, while 
by~\ref{prop:4.4}.3$^\circ$ we have $\|\tilde{T_n}\|_{L^qM \rightarrow L^{q'}M}=\|x_n\|_{p/2} \rightarrow 0.$ 

\end{proof}

The next result, which is crucial in proving Theorem~\ref{thrm:6.5} later in this paper (Theorem~\ref{thrm:1.1} in the introduction) should be compared to \cite{O10} and Proposition 3.1 in \cite{DKEP22} where similar decompositions are considered. The previous corollary shows however that the proof strategy employed there will not apply to our current situation.

\begin{lemma}\label{lemma:4.6}
For any $T \in \B(L^2 M)$ 
$$\vertiii{T}_p  = \inf\{ \|a\|_p \|S\| \|b\|_p: a,b \in M,S\in \B(L^2M), aSb = T
\}.
$$
\end{lemma}
\begin{proof}
Using Lemma~\ref{lemma:4.2}, we have that for any decomposition $T = aSb$ with $a,b \in M$ and  $S \in \B(L^2M)$ that $\vertiii{T}_p\leq \|a\|_p \|S\| \|b\|_p$. This shows at the very least that $\vertiii{T}_p$ is smaller than this infimum. 

To obtain the reverse inequality, we use a convexity argument. To this end, for the remainder of the proof we consider the following subsets of $\B(L^2M)$. For any positive number $\alpha$, let $\C_{\alpha}$ be the set of operators $T \in \B(L^2M)$ such that we can find a decomposition $T = aSb$ with $a,b \in M$ and  $S \in \B(L^2M)$ such that $\|a\|_p \|S\| \|b\|_p \leq \alpha$. We claim first that  $\C_\alpha$ is convex.

For let's say we have operators $T_1,T_2 \in \C_\alpha$. Then by definition, we can find decompositions $T_1 = a_1 S_1 b_1$ and $T_2 = a_2 S_2 b_2$ with the $a_i, b_i \in M$ and the $S_i \in \B(L^2 M)$ such that $\|a_i\| \|S_i\| \|b_i\|_p \leq \alpha$. After rescaling, we may assume without loss of generality that $\|S_i\| = 1$ and $\|a_i\|_p = \|b_i\|_p \leq \alpha^{1/2}$. For any $\lambda \in(0,1)$, we can form the decomposition $\lambda T_1 +(1-\lambda)T_2 = aSb$ as follows. First, factor $\lambda T_1 +(1-\lambda)T_2$ as a chain of operators through $L^2M \oplus L^2M$ by noting 
$$\lambda a_1S_1b_1 + (1-\lambda)a_2S_2b_2 
=
\begin{pmatrix}
\lambda^{1/2} a_1 & (1-\lambda)^{1/2}a_2
\end{pmatrix}
\begin{pmatrix}
S_1 & 0 \\
0 & S_2
\end{pmatrix}
\begin{pmatrix}
\lambda^{1/2}b_1 \\
(1-\lambda)^{1/2} b_2
\end{pmatrix}.
$$
Let $au$ be the polar decomposition of $(\lambda^{1/2} a_1 \text{ } (1-\lambda)^{1/2}a_2)$, where $a\in M$ is positive, and let $ub$ be the polar decomposition of ${\lambda^{1/2}b_1 \choose (1-\lambda)^{1/2}b_2}$, where here $b \in M$ is positive. If we call $S \in \B(L^2M)$ the product 
$$
S = 
u \begin{pmatrix}
S_1 & 0 \\
0 & S_2
\end{pmatrix}
v.
$$
We arrive at a decomposition $T_1 + T_2 = aSb$. 

Now we can calculate that 
$$aa^* = \lambda a_1a_1^*  + (1-\lambda) a_2a_2^*.$$ 
So that 
\begin{equation}
\begin{split}
\|a\|_p^2 = \|aa^*\|_{p/2} 
&\leq \lambda\|a_1a_1^* \|_{p/2}+(1-\lambda)\|a_2a_2^* \|_{p/2}
\\
&= \lambda\|a_1\|_p^2+(1-\lambda)\|a_2\|_p^2
\\
&\leq \alpha.
\end{split}
\end{equation}
By the same logic we also have 
$\|b\|_p^2 \leq \alpha
$.
Lastly, we see by inspection that $\|S\| = \max\{\|S_1\|,\|S_2\| \} = 1$. Putting this together, we have then
\begin{equation}
\begin{split}
\|a\|_p \|S\| \|b\|_p
&\leq 
\alpha
\\
\end{split}
\end{equation}
It follows then that $\lambda T_1 +(1-\lambda)T_2 \in \C_\alpha$, so $\C_\alpha$ is convex. 

Next, let's say we have an operator $T \in \B(L^2M)$ with $\vertiii{T}_p = \alpha$. We claim that $T$ is in the $\vertiii{\cdot}_p$ norm closure of $\C_\alpha$. For otherwise, since $\C_\alpha$ is convex, we can find using Hahn-Banach a functional $\varphi \in \B^*(p)$ with $\|\varphi\|_{\B^*(p)} = 1$ such that 
$$\text{Re}(\varphi(T)) 
> 
\sup_{S \in \C_\alpha} \text{Re}(\varphi(S)).
$$
Now, by definition we have 
$$\text{Re}(\varphi(T))  \leq \vertiii{T}_p = \alpha.$$
Moreover, since $\|\varphi\|_{\B^*(p)}$ was equal to the supremum of all $|\varphi(aSb)|$ where $\|a\|_p, \|b\|_p \leq 1$ and $\|S\| \leq 1$, it follows that 
$$\sup_{S \in \C_\alpha} \text{Re}(\varphi(S)) = \alpha.$$
 But this  leads to a contradiction, so we must have $T \in \overline{\C_{\alpha}}^{\vertiii{\cdot}_p}$.
 
On the other hand, we claim that the closure of $\C_\alpha$ in the $\sigma(\B(L^2M), \B^*(p))$ topology is $\bigcap_{\beta> \alpha} \C_\beta$. For consider  $$p(T) = \inf\{\beta: T \in \C_\beta\}$$
 the seminorm corresponding to the convex sets $\C_\beta$. Note that a linear functional $\varphi$ on $\B(L^2M)$ will be bounded with respect to the seminorm $p$ if and only if $\varphi$ is bounded on any set $\C_\beta$. But as we observed already $\sup_{S \in \C_{\beta}} |\varphi(S)|= \beta \|\varphi\|_{\B^*(p)}$, so the only such functional are in $\B^*(p)$. It follows then that the $\sigma(\B(L^2M), \B^*(p))$ topology and the weak topology for $(\B(L^2M), p)$ coincide. Hence, since $\C_\alpha$ was convex, the closure of $\C_\alpha$ in the $\sigma(\B(L^2M), \B^*(p))$ topology will be the same as the closure of $\C_\alpha$ with respect to the seminorm $p$, which is indeed $\bigcap_{\beta> \alpha} \C_\beta$.

To complete the proof, we notice by convexity that the $\vertiii{ \, \cdot \, }_p$ norm closure and the $\sigma(\B(L^2M), \B^*(p))$ closure of $\C_\alpha$ are the same. Hence, for any $T \in \B(L^2M)$ such that $\vertiii{T}_p = \alpha$ we have 
$$T \in 
\overline{\C_{\alpha}}^{\vertiii{\cdot}_p}
=
\overline{\C_{\alpha}}^{\sigma(\B(L^2M), \B^*(p))}
=
\bigcap_{\beta> \alpha} \C_\beta.
$$
It follows the that we can find decomposition $T = aSb$ with $\|a\|_p\|S\| \|b\|_p$ arbitrarily close to $\alpha$, and the lemma follows immediately.

\end{proof}

\begin{lemma}\label{lemma:4.7}
Assume $\{u_n\}_n\subset \U(M)$ is a sequence of unitary elements in $M$ with $\tau(u_n^*u_m)=0$ for all $n\neq m$. For each $T\in \B(L^2M)$ 
let $\E_0(T) \in \B$ be the operator that acts as $0$ on $
\mathcal{H}_0^\perp$, where $\mathcal{H}_0=\overline{sp} (\{\hat{u}_n\}_n)\subset L^2M$, and as the diagonal operator 
that takes $\hat{u}_n$ to $\langle T(\hat{u}_n), \hat{u}_n \rangle \hat{u}_n$. Then $\vertiii{\E_0(T)}_p \leq \vertiii{T}_p$, for all $2\leq p \leq \infty$. 
Moreover, if $T\in \B(L^2M)$ is diagonal with respect to $\{\hat{u}_n\}_n$, i.e., $\E_0(T)=T$, then $\vertiii{T}_p$ is equal to the operator norm of $T$ in $\B(L^2M)$, $\forall p$. 
\end{lemma} 
\begin{proof} By Proposition~\ref{prop:3.2}, one has  $\vertiii{T}_p \leq \|T\|$, $\forall T\in \B(L^2M)$. If in addition $T$ is diagonal with respect to $\{\hat{u}_n\}_n$ and equal to $0$ on $\mathcal H_0^\perp$, 
then $\|T\|=\sup_n |\langle T(\hat{u}_n), \hat{u}_n\rangle |$.  But by Lemma~\ref{lemma:3.4} and the definition of $\vertiii{T}_p$, the right hand side is larger than or equal to $\vertiii{T}_p$, showing that  
$\vertiii{T}_p \geq \|T\|$ as well, so altogether $\vertiii{T}_p=\|T\|$.  

For an arbitrary $T\in \B(L^2M)$, by the definition of $\vertiii{T}_p$ and Lemma~\ref{lemma:3.4} one has $\vertiii{T}_p\geq |\langle T(\hat{u}_n), \hat{u}_n \rangle|$, $\forall n$. 
Thus, $\vertiii{T}_p \geq \sup_n |\langle T(\hat{u}_n), \hat{u}_n\rangle |=\|\mathcal E_0(T)\|=\vertiii{\mathcal E_0(T)}_p$.  

\end{proof}

\begin{cor} If $M=L\Gamma$ and we denote $\D=\ell^\infty\Gamma \subset \B(\ell^2\Gamma)=\B(L^2M)$, $\C=c_0(\Gamma) \subset \ell^\infty\Gamma$, 
then for each $T\in \D$ we have $\vertiii{T}_p=\|T\|$, $\forall p\geq 2$. Thus, $\C \subset \D$  are $\vertiii{ \, \cdot \, }_p$-closed in $\B(p)$. 

\end{cor} 

\begin{proof} Since for $M=L\Gamma$ we have $L^2M=\ell^2\Gamma$, with $\{\hat{u_g}\}_g$ as orthonormal basis, the previous lemma 
implies that $\vertiii{ \, \cdot \, }_p$ restricted to $\D=\ell^\infty\Gamma$ coincided with the operator norm.

\end{proof}

\section{The Banach bimodules $\K(p)$, $2\leq p < \infty$ }

Since the ideal of compact operators $\K(L^2M)$ is a Banach bimodule over both $M, M^{op}$, its $\vertiii{ \, \cdot \, }_p$-completions, $2\leq p < \infty$,  
give rise to a one parameter family of bimodules that we now consider. 

\begin{defn}
 For each $2\leq p < \infty$,  we denote by $\K(p)$ 
the closure of $\K=\K(L^2M)$ in $\B(p)$. 
\end{defn}

\begin{lemma}\label{lemma:5.2} Let $2\leq p < \infty$ and denote $q=\frac{2p}{p-2}$ and $q'=\frac{2p}{p+2}$ as before. 
Following Proposition~\ref{prop:4.4}, for each $K\in \B(p) $ we denote by $\tilde{K}$ the element it induces  in $ \B(L^qM, L^{q'}M)$. 
\vskip.05in
$1^\circ$ If $K \in \K(p)$, then  $\tilde{K}$  takes the unit ball $(L^qM)_1$ into a $\| \cdot \|_{q'}$-compact subset of $L^{q'}M$. 
\vskip.05in 

$2^\circ$ If $K\in \K(p)$, then for any sequence of unitary elements $\{u_n\}_n \subset M$ that converges weakly to $0$, one has $\|\tilde{K}(\hat{u}_n)\|_{q'}\rightarrow 0$. 
\end{lemma}
\begin{proof}
$1^{\circ}$ Note first that if $K \in \K(p) $ is a finite-rank operator, then $\tilde{K}$ is in $\K(L^qM, L^{q'}M)$. If $K$ is a general element of $\K(p)$, then we can find a sequence $(K_n)$ of finite-rank operators in $\B(p)$ such that $\K_n\to K$ with respect to the $\vertiii{ \, \cdot \, }_p$ norm. Since the mapping $K \mapsto \tilde{K}$ is contractive, 
we have $\tilde{K_n} \to \tilde{K}$ in $\B(L^qM, L^{q'}M)$. Since the space $\K(L^qM, L^{q'}M)$ is closed in $\B(L^qM, L^{q'}M)$, 
and each finite-rank $K_n$ was in $\K(L^qM, L^{q'}M)$, it follows  that $K \in \K(L^qM, L^{q'}M)$ as well.

\vskip.05in

$2^{\circ}$
As in part $1^{\circ}$, if $K \in \K(p)$ is a assumed to be a finite-rank operator the claim follows immediately. Let now $K \in \K(p)$ be arbitrary and let $(K_n)_n$ be a sequence of finite-rank operators in $\B(p)$ such that  $K_n \to K$ with respect to the  $\vertiii{ \, \cdot \, }_p$ norm. Now if $\epsilon >0$ we can find an $m$ such that 
$$\|\tilde{K} - \tilde{K_m}\|_{L^q M \to L^{q'} M} \leq \vertiii{K - K_m}_p < \epsilon/2.$$
Now since $K_m$ is finite-rank, we can also find an $N$ such that for all $n \geq N$ we have $\| K_m(\hat{u}_n) \|_{q'} <\epsilon$. Combining these we get for any $n \geq N$
 
\begin{equation}
\begin{split}
\|K(\hat{u}_n )\|_{q'}
&\leq 
\|(K-K_m)(\hat{u}_n) \|_{q'}
+
\|K_m(\hat{u}_n) \|_{q'}
\\
&<
\| \hat{u}_n \|_{q} \|K - K_m\|_{L^q M \to L^{q'} M} 
+
\epsilon /2
\\
&< \epsilon.
\end{split}
\end{equation}
Thus we conclude $\|\tilde{K}(\hat{u}_n )\|_{q'} \to 0$ as desired.

\end{proof}

\begin{prop} For each $2\leq p < \infty$, $\K(p)$ endowed with the norm $\vertiii{ \, \cdot \, }_p$ is a Banach $M$-bimodule and Banach $M^{op}$-bimodule.

\end{prop}
\begin{proof} The fact that $\K(p)$ is a Banach space is clear from the fact that it is a norm closed subspace of $\B(p)$. The $M$-bimodule and $M^{op}$-bimodule structure also follows by restricting from $\B(p)$, and by taking into account that  the $\vertiii{ \cdot }_p$-dense subset $\K(L^2 M)$ of $\K(p)$ 
is invariant under left and right multiplication by elements of $M$, $M^{op}$.

\end{proof}

Recalling that $\K(L^2M)^*  = \B_{\text{\rm n}}^*$, we now prove the analogous result for the spaces $\K(p)$ and $\B^*_{\text{\rm n}}(p)$.

\begin{thrm} For all $p\geq 2$, $\K(p)^*=\B_{\text{\rm n}}(p)$. Also, for each $2\leq p < \infty$, $\K(p)$ endowed with its norm $\vertiii{ \, \cdot \, }_p$  is a smooth $M$-bimodule, in the sense of definition~\ref{sec:2.4}. 
\end{thrm}
\begin{proof}
We use the description of preduals in \cite{K77} to prove the result. Note first that the compact operators $ \K(L^2M)$, form a subspace of $(\B_{\rm n}(p))^*$.  Furthermore, the space $\K(L^2M)$ separates points of $\B_{\rm n}(p)$, i.e.\ for any distinct $\varphi, \varphi' \in \B_{\rm n}(p)$ there is a $K \in \K(L^2M)$ such that $\varphi(K) \neq \varphi'(K)$. This is because the space $\K(L^2M) $, which is strictly smaller than $\K(p)$, separates the points of its dual $\B(L^2 M)_*$, which is strictly larger than $\B_{\rm n}(p)$. 

Now we claim that the unit ball of $\B^*_{\text{\rm n}}(p)$ is compact in the $\sigma(\B^*_{\text{\rm n}}(p), \K(L^2 M))$ topology. To see this, consider a net $(\varphi_{\alpha})$ in the unit ball $(\B^*_{\text{\rm n}}(p))_1$. Recall that the norm $\|\cdot \|_{\B^*(p)}$  majorizes the usual norm on $\B^*$ and that $( \B^*_{\text{\rm n}}(p))_1 $ is a subset of $(\B^*_{\text{\rm n}})_1$. Since the predual of the space $\B^*_{\text{\rm n}}$ of normal linear functionals on $\B$ is $\K(L^2M)$, we have that $(\B^*_{\text{\rm n}})_1$ is  compact in the $\sigma(\B^*_{\text{\rm n}}, \K(L^2 M))$ topology. Thus, there exists a subnet $(\varphi_{\beta})$ of our original net that converges to some $\varphi \in (\B^*_{\text{\rm n}})_1$ in the $\sigma(\B^*_{\text{\rm n}}, \K(L^2 M))$ topology. 

We claim further that this $\varphi$ is actually in $\B^*_{\text{\rm n}}(p)$. Let $x$ and $y$ be any elements of $M$. 
Since the space $\B^*_{\text{\rm n}}$ is a dual normal Banach $M$-bimodule, it follows the net $x\cdot \varphi_{\beta} \cdot y$ converges to $x \cdot \varphi\cdot y$ in the $\sigma(\B^*_{\text{\rm n}}, \K(L^2 M))$ topology. Note that this implies $\|x \cdot \varphi\cdot y\| \leq \sup_{\beta} \|x \cdot \varphi_{\beta}\cdot y\| $. In particular, notice that if $x$ and $y$ are chosen so that $\|x\|_p , \|y\|_p \leq 1$, then, since the net $(\varphi_{\beta})$ lies in $(\B^*_{\text{\rm n}}(p))_1$, we would have $\|x \cdot \varphi\cdot y\| \leq \sup_{\beta} \|x \cdot \varphi_{\beta}\cdot y\| \leq 1 $. Varying over all $x$ and $y$ with $p$-norm less than 1, we gather that $\|\varphi\|_{\B^*(p)} \leq 1$, so indeed our $\varphi$ lies in $(\B^*_{\text{\rm n}}(p))_1$. It follows then that  the unit ball $(\B^*_{\text{\rm n}}(p))_1$ is compact in the  $\sigma(\B^*_{\text{\rm n}}(p), \K(L^2 M))$ topology. Using the description of preduals in [K77], we find that a predual of  $\B^*_{\text{\rm n}}(p)$ is the norm closure of $\K(L^2M)$ in the dual space $(\B^*_{\text{\rm n}}(p))^*$. By definition this is  $\K(p)$. 

It remains then to check that $\K(p)$ is a smooth bimodule. Note that if $T \in \K(p)$ lies in $\B(L^2M)$, then by Lemma~\ref{lemma:4.2} 
the maps $x \mapsto Tx$ and $x \mapsto xT$ are $\| \cdot \|_2$ to $\vertiii{ \, \cdot \, }_p$ continuous on the unit ball of $M$.  If $T \in \K(p)$ is  an  arbitrary element, 
then we for any $\epsilon >0$ there exists $S \in \K(p) \cap \B(L^2M)$ such that $\vertiii{T-S}_p < \epsilon$. But then for any net $(x_\iota)$ in the unit ball of $M$ such that $\|x_\iota\|_2 \to 0$ we have 
$$\vertiii{Tx_\iota}_p \leq \vertiii{(T-S)x_\iota}_p + \vertiii{Sx_\iota}_p < \epsilon + \vertiii{Sx_\iota}_p.$$
Hence $\limsup_{\iota} \vertiii{Tx_\iota}_p \leq \epsilon$. Since $\epsilon$ was arbitrary, it follows $\vertiii{Tx_\iota}_p$ tends to 0. Hence the map $x \mapsto Tx$ is still $\| \cdot \|_2$ to $\vertiii{ \, \cdot \, }_p$ continuous on the unit ball of $M$. A similar argument shows the same for the map $x \mapsto xT$. It follows then that $\K(p)$ is smooth.

\end{proof}

\section{The $\text{\rm q}_M$-topology and the bimodule $\text{\rm q}\mathcal K_M$} 

In this section we consider a new topology on Banach bimodules over tracial von Neumann algebras $(M, \tau)$, 
which we will denote $\text{\rm q}_M$, that takes into consideration 
the trace on $M$, and which we will refer to as the {\it $\tau$-rank topology} (sometimes also called the topology of convergence in measure). When applied to 
the Banach $M$-bimodule $\B(L^2M)$, the restriction of the $\text{\rm q}_M$-topology 
to the unit ball $(\B(L^2M))_1$ is ``almost the same'' as the topology given by $\vertiii{ \, \cdot \, }_p$-norms, but finer. 
However, the $\text{\rm q}_M$-closure in $\B(L^2M)$ of the unit ball of compact operators $(\K(L^2M))_1$ 
coincides with its $\vertiii{ \, \cdot \, }_p$-closure, thus giving rise to an interesting Banach $M$-bimodule 
of ``almost-compact'' operators denoted $\text{\rm q}\K_M$.

\begin{defn}
 Let $\B$ be a Banach $M$-bimodule. We say that a net $(T_i)_i\subset \B$ is $\text{\rm q}_M$-convergent 
to $T\in \B$ if  the following conditions are satisfied: $\sup_i \|T_i\|<\infty$; for any $\varepsilon>0$,   
there exists $i_0$ such that for any $i\geq i_0$ there exists a projection $p\in \mathcal{P}(M)$ with 
$\tau(1-p) < \varepsilon$, $\|p(T_i-T)p\| < \varepsilon$. 
\end{defn}

Note that if these conditions are satisfied, then $\|T\| \leq \limsup_i \|T_i\|$. 
Thus, for any finite $r>0$, the $\text{\rm q}_M$-convergent nets in $(\B)_r$ 
define a topology on $(\B)_r$, that we will also denote by $\text{\rm q}_M$. Note also that if $r'\geq r >0$, then the 
restriction to $(\B)_r$ of the $\text{\rm q}_M$-topology on $(\B)_{r'}$, coincides with the $\text{\rm q}_M$-topology on $(\B)_r$. 

Note that the $\text{\rm q}_M$-topology on any bounded subset of $\B$ is implemented by the metric given by $\text{\rm q}_M(T,S)=
\inf \{\tau(1-p) + \|p(T-S)p\| \mid p\in \mathcal{P}(M)\}$. 

Given a linear subspace $\B_0\subset \B$, we denote $\overline{\B_0}^{\text{\rm q}_M}$ the union over all $r>0$ 
of the $\text{\rm q}_M$-closures of $(\B_0)_r$ in $(\B)_r$. Equivalently, $\overline{\B_0}^{\text{\rm q}_M}$ is the set 
of all scalar multiples of elements in $\overline{(\B_0)_1}^{\text{\rm q}_M}.$

\vskip .05in

The $\text{\rm q}_M$-topology on $(\B)_r$ is obviously weaker than the norm topology. 
A typical example of a Banach $M$-bimodule $\B$ that we consider is the algebra $\B(\mathcal{H})$ of all linear bounded 
operators on a Hilbert space $\mathcal{H}$ on which $M$ acts normally and faithfully, with the $M$-bimodule structure 
given by left-right multiplication by elements in $M$. More generally, we consider (linear) subspaces $\B\subset \B(\mathcal{H})$ 
with $M\B M\subset \B$, such as the space of compact operators $\K(\mathcal{H})$ on $\mathcal{H}$. 
For this class of examples, another natural topology on $\B$ is the s$^*$-topology. 
If $\B=M$, then this is easily seen to coincide with the $\text{\rm q}_M$-topology  on bounded sets. 
But in general, the s$^*$-topology is strictly weaker  than the $\text{\rm q}_M$-topology 
on $(\B)_1$ (notably if $\B=\K(\mathcal{H})$ and $M$ 
is infinite dimensional, see below). 

\begin{prop}\label{prop:6.2}
 Let $\B$ be a dual normal $M$-bimodule and $\B_0 \subset \B$ a 
norm closed sub-bimodule. 

\vskip .05in

$1^\circ$ For any $T\in \B$, the maps $(M)_1 \ni x \mapsto xT, Tx \in \B$ are $\| \cdot \|_2-\text{\rm q}_M$ continuous. 

\vskip .05in 

$2^\circ$ $(\B)_1$ is complete in the $\text{\rm q}_M$-metric $($and thus so is $\overline{(\B_0)_1}^{\text{\rm q}_M} \subset (\B)_1)$.

\vskip .05in 

$3^\circ$ $\overline{\B_0}^{\text{\rm q}_M}$ is a Banach $M$-bimodule. 

\vskip .05in 

$4^\circ$ Given any norm-separable subspace $\E \subset \overline{\B_0}^{\text{\rm q}_M}$, there exists an 
increasing sequence of projections $p_n\in M$ with $p_n \rightarrow 1$ such that $p_nTp_n \in \B_0$, 
for all $T\in \E$. 
\end{prop}

\begin{proof} $1^\circ$ If $\varepsilon >0$ and $\|x\|\leq 1$ satisfies $\|x\|_2\leq \varepsilon$, 
then the spectral projection $p$ of $xx^*$ corresponding to the interval $[0,\varepsilon]$ has trace at least 
$1-\varepsilon$, or else we have $\|x\|_2^2=\|px\|_2^2+\|(1-p)x\|_2^2 > \|(1-p)x\|_2^2\geq \varepsilon^2$, 
a contradiction. Thus, we have $\|pxT\|\leq \|px\| \|T\|\leq \varepsilon$ and $\tau(1-p)\leq \varepsilon$. 
This shows that $(M)_1 \ni x \mapsto xT \in \B$ is $\| \cdot \|_2-\text{\rm q}_M$ continuous. The proof for  
$(M)_1 \ni x \mapsto Tx \in \B$ is similar. 

$2^\circ$ If $T_n\in (\B)_1$ is $\text{\rm q}_M$-Cauchy, then for any $k \geq 1$, there exists 
$n_k$ such that for any $n, m \geq n_k$ there exists a projection $p_k\in \mathcal{P}(M)$ 
with the property that $\tau(1-p_k)+\|p_k(T_m-T_n)p_k\|\leq 2^{-k}$. Thus, the sequence of projections $P_k=\wedge_{l\geq k} p_l$, $k\geq 1$, 
is increasing and satisfies $\tau(1-P_k)\leq 2^{-k+1}$, $\|P_k(T_n-T_m)P_k\|\leq 2^{-k}$ for any $n, m \geq n_k$. By the inferior semicontinuity of the 
norm on $\B$ with respect to the $w^*$-topology, it follows that any $w^*$-limit point $T\in (\B)_1$ of the sequence $\{T_m\}_m$ satisfies 
$\|P_k(T-T_n)P_k\|\leq 2^{-k}$ for any $n\geq n_k$. This shows that $\{T_m\}_m$ is $\text{\rm q}_M$-convergent to $T$. 
  
$3^\circ$ If $T_n$ is a sequence in $\overline{\B_0}^{\text{\rm q}_M}$ that converges in norm to some $T\in \B$, 
then $T_n$ is automatically bounded and by $2^\circ$ we have $T \in \overline{\B_0}^{\text{\rm q}_M}$ as well. 
The invariance of $\overline{\B_0}^{\text{\rm q}_M}$ to left-right multiplication by elements in $M$ is obvious. 

$4^\circ$ It is sufficient to show the existence of such projections for a countable subset   
$\{T_i\}_i \subset \overline{(\B_0)_1}^{\text{\rm q}_M}$. For each $i\leq n \leq k$, there exists a projection $p_{i,k}\in M$ 
and $S_{i,k}\in (\B_0)_1$ such that $\|p_{i,k}(T_i-S_{i,k})p_{i,k}\|\leq 2^{-k}$ and $\tau(1-p_{i,k})\leq 2^{-k}/n$. 
Thus, if we let $P_{n,k}=\wedge_{i\leq n} p_{i,k}$, 
then $\|P_{n,k}(T_i-S_{i,k})P_{n,k}\|\leq 2^{-k}$, $\forall i\leq n$, and $\tau(1-P_{n,k})\leq 2^{-k}$. If we now put $P_n=\wedge_{k\geq n} P_{n,k}$, 
then $P_n$ is increasing, $\tau(1-P_n)\leq 2^{-n+1}$, and $\|P_n(T_i-S_{i,k})P_n\| \leq 2^{-n}$, $\forall i\leq n\leq k$. For each fixed 
$m$, by applying this to $k=n\geq m$ and taking into account that $P_mP_n=P_m$, 
it follows that $\|P_m(T_i-S_{i,n})P_m\| \leq 2^{-n}$, $\forall i\leq m$. This shows in particular that $\{P_mS_{i,n}P_m\}_n\subset (\B_0)_1$ is 
norm-Cauchy and thus convergent to some $X_{i,m}\in (\B_0)_1$. It also shows that $X_{i,m}$ satisfy $P_mT_iP_m=X_{i,m}$, while  
$\tau(1-P_m)\leq 2^{-m+1}$, $\forall i\leq m$. 

\end{proof}

\begin{defn}
 Given a tracial von Neumann algebra $(M, \tau)$ in its standard representation 
 on $L^2M$, we denote by $\text{\rm q}\K_M$ the $\text{\rm q}_M$-closure of $\K(L^2M)$ in $\B(L^2M)$ 
 and call its elements $\text{\rm q}_M$-{\it compact operators}.   
 \end{defn}
 
 Notice that besides its $M$-bimodule structure, the algebra $\B(L^2M)$ also has an $JMJ=M'$ 
 bimodule structure, where $J: L^2M \rightarrow L^2M$ is the canonical conjugacy defined 
 by $J(x)=x^*$, $x\in M\subset L^2M$, and $M'$ denotes as usual the commutant of $M$ in $\B(L^2M)$. The 
 algebra $JMJ=M'$ can be naturally identified with the opposite algebra $M^{\text{\rm op}}$ of $M$, 
 and we will retain this notation for $JMJ$.

\begin{prop}
 The space $\text{\rm q}\K_M$ is a norm closed $^*$-subspace of $\B(L^2M)$,  
 which is both an $M$-bimodule and an  $M^{\text{\rm op}}$-bimodule.  
\end{prop}
\begin{proof}
By applying Proposition~\ref{prop:6.2} to $\B_0=\K(L^2M)\subset \B(L^2M)=\B$, 
it follows that $\text{\rm q}	_M$ is a norm closed $M$-bimodule. It is clearly an $M^{\text{\rm op}}$-bimodule 
and closed under the $^*$-operation. 

\end{proof}

\begin{thrm}\label{thrm:6.5} For each $2 \leq p < \infty$ denote by  $\K_p$ the space of all 
operators $T\in \B=\B(L^2M)$ with the property that there exists a sequence  of compact operators $K_n \in \K_p$ such that 
$\sup_n \|K_n\|<\infty$ and $\lim_n \vertiii{T-K_n}_p=0$. Then $\K_p = \text{\rm q}\K_M$.  
\end{thrm}
\begin{proof} To see that $\text{\rm q}\K_M \subset \K_p$  
let us show that the $\text{\rm q}_M$-topology on the unit ball of $\B(L^2M)$ is stronger than the $\vertiii{ \, \cdot \, }_p$-topology, $\forall 2\leq p < \infty$. Indeed, 
by Lemma~\ref{lemma:4.2}, if $T\in \B(L^2M)$ and $P\in \mathcal P(M)$, then we have 
\begin{equation}
\begin{split}
\vertiii{T}_p 
&\leq \vertiii{PTP}_p + \vertiii{PT(1-P)}_p + \vertiii{(1-P)T}_p 
\\
&\leq 
\|PTP\| + 2\|1-P \|_p \|T\|= \|PTP\| + 2(\tau(1-P))^{1/p} \|T\|. 
\end{split}
\end{equation}
This shows that $\vertiii{T}_p \leq 2 \inf \{\|PTP\|  + (\tau(1-P))^{1/p} \|T\|\}$, implying that the $\text{\rm q}_M$-toplogy on $(\B(L^2M))_1$ is stronger 
than the $\vertiii{ \, \cdot \, }_p$-topology.

To show that $\K_p \subset \text{\rm q}\K_M$, let $T$ be an operator in $K_p$. 
Then we can find a sequence of uniformly bounded compact operators $K_n \in \K(L^2M)$ such $\vertiii{T-K_n}_p$ tends to 0. Since $T-K_n$ is still in $\B(L^2M)$, Lemma~\ref{lemma:4.6} says we can find $a_n,b_n \in M$ and $S_n \in B(L^2M)$ such that $T-K_n = a_nS_nb_n$ and $\|a_n\|_p \|S_n\| \|b_n\|_p $ tends to 0. Taking spectral projections of $|a_n|$ and $|b_n^*|$ we can find a sequence of projections $p_n$ with $\|p_n\|_p$ tending to 1 such that 
$$\|p_n(T-K_n)p_n \| = \|p_na_n S_n b_n p_n\| \to 0$$

Thus, $\K_p \subset \text{\rm q}\K_M$, which combined with the first part shows that $\K_p=\text{\rm q}\K_M$.

\end{proof}

It is useful to note that due to their ``compact nature'', elements in the spaces $\text{\rm q}\K_M$ cannot intertwine 
diffuse subalgebras of $M$. This fact will be used later to deduce that an operator in $\B(L^2M)$ 
that commutes with $M$ modulo $\text{\rm q}\K_M$ and commutes with a diffuse subalgebra of $M$, must in fact commute with all of $M$. 

\begin{lemma}\label{lemma:6.6} Let $B\subset eMe$ be a diffuse von Neumann
subalgebra and $\sigma: B \rightarrow fMf$ be a unital faithful
$^*$-homomorphism, for some non-zero projections $e,f\in M$. If $K\in \text{\rm q}\K_M$
satisfies $Kb=\sigma(b)K$, $\forall b\in B$, then $K=(1-f)K(1-e)$.
\end{lemma}
\begin{proof}
 Note that $fK(1-e)=0$ and $(1-f)Ke=0$. By
replacing $K$ by $K-(1-f)K(1-e)$, we may also assume
$(1-f)K(1-e)=0$. So we have to prove that if $K$ satisfies the condition in the hypothesis and $K=fKe$, then $K=0$.

Let $u$ be a Haar unitary in $B$ and $x\in eM$. Since $u^nx$ tends weakly to $0$ and $\|v(\xi)\|_1=\|\xi\|_1$ 
for any unitary $v\in fMf$ and $\xi\in L^1(fM)$, we get

$$
0=\lim_n \|K(\widehat{u^nx})\|_1=\lim_n\|\sigma(u^n) (K(\hat{x}))\|_1
= \|K(\hat{x})\|_1, $$
where the first equality follows easily from the definition of $\text{\rm q}\K_M$.
This shows that $K=fKe$ satisfies $K(\widehat{eM})=0$, thus $K=0$. 

\end{proof}

\section{Derivations of $M$ into  $\text{\rm q}\K_M$ }

Recall that if $M$ is a Banach algebra (always assumed unital) and $\B$ is a Banach $M$-bimodule, then a derivation of $M$ into $\B$ 
is a linear map $\delta: M \rightarrow \B$ satisfying the property $\delta(xy)=x\delta(y)+ \delta(x)y$, for all $x, y \in M$. 

It is immediate to check 
that if $T\in \B$, then the map $\text{\rm ad}T: M \rightarrow \B$ defined by $\text{\rm ad}T(x)=[T, x]:=Tx-xT$, $x\in \B$, is a derivation. 
Such derivations are called {\it inner}. 

It is useful to note that if $F\subset M$ is a set, then $\delta_{|F}$ determines the values of $\delta$ on all the 
algebra $\text{\rm Alg}(F)$ generated by $F$. 

Recall from \cite{R72} that 
a derivation  is automatically norm-continuous. Moreover, if $M$ is a von Neumann algebra and $\B$ is a dual normal $M$-bimodule, then 
any derivation is automatically continuous from $M$ with the ultra-weak topology to $\B$ with its $\sigma(\B, \B_*)$ topology. 

Thus, if $F=F^* \subset M$ is a set that generates $M$ as a von Neumann algebra and $M_0$ is the norm closure of the $^*$-algebra generated by $F$, 
then any derivation $\delta$ of $M$ into a Banach $M$-bimodule  is uniquely determined on $M_0$ by the values it takes on $F$, $\delta_{|F}$. If in addition 
$\B$ is a dual normal Banach bimodule, then all of $\delta$ is uniquely determined by $\delta_{|F}$. 

Let us first notice an automatic continuity (smoothness) result for derivations, with 
respect to the $\text{\rm q}_M$-metric and the $\vertiii{ \, \cdot \, }_p$ norms.

\begin{thrm}\label{thrm:7.1}
 Let $(M,\tau)$ be a tracial von Neumann algebra,  $\B$ a Banach $M$-bimodule and 
$\delta: M \rightarrow \B$ a derivation. 
Then $\delta$ is automatically $\| \cdot \|_2$-$\text{\rm q}_M$ continuous on $(M)_1$. More precisely, if $\varepsilon >0$, then given any $x\in (M)_1$ 
with $\|x\|_2\leq (\varepsilon/2)^{3/2}$, there exists $p\in \mathcal P(M)$ such that $\tau(1-p)\leq \varepsilon$ and $\|p\delta(x)p\|\leq \varepsilon \|\delta\|$.

In particular, if 
$\B = \text{\rm q}\K_M$, then $\delta$ is automatically continuous from $(M)_1$ with the $\| \cdot \|_2$-topology to 
$\text{\rm q}\K_M$ with the topology given by the $\text{\rm q}_M$-metric. 
\end{thrm}
\begin{proof}
By [R72], $\delta$ is automatically norm continuous and without loss of 
generality we may assume $\|\delta\|=1$. Let $\varepsilon >0$. Let $x\in (M)_1$ be so that $\|x\|_2\leq (\varepsilon/2)^{3/2}$. 
Denote by $e$ the spectral projection of $xx^*$ corresponding to $[0,\varepsilon^2/4]$. Then $e$ satisfies $\|ex\|\leq \varepsilon/2$ and 
$(1-e)xx^* \geq (\varepsilon/2)^2(1-e)$. Thus we have: 
$$
(\varepsilon/2)^3 \geq \|x\|_2^2=\|ex\|_2^2+\|(1-e)x\|_2^2 
$$
$$
\geq  \|(1-e)x\|_2^2=\tau((1-e)xx^*) \geq (\varepsilon/2)^2\tau(1-e). 
$$ 
This implies 
that $\tau(1-e) \leq \varepsilon/2$. Similarly, if $e'$ denotes the spectral projection of $x^*x$ corresponding to $[0,\varepsilon^2/4]$, 
we have $\tau(1-e')\leq \varepsilon/2$ and $\|xe'\|\leq \varepsilon/2$. Thus, if we denote $p=e\wedge e'$, then $\tau(1-p)\leq \varepsilon$ and 
$\|p\delta(x)p\|=\|\delta(px)p-\delta(p)xp\|\leq \varepsilon$. 

\end{proof}

\begin{lemma}\label{lemma:7.2}  Assume $T\in \B(L^2M)$ is so that 
$[T, M_0]\subset \text{\rm q}\K_M$ for some weakly dense $^*$-subalgebra $M_0\subset M$. Then we have: 

\vskip.05in 

$1^\circ$ $[T, M] \subset  \text{\rm q}\K_M$. 

\vskip.05in

$2^\circ$ If, in addition, $T=e$ is a projection and there exists a Haar unitary $u\in M$ such that $[e, u]\in \K(L^2M)$ 
with $eue$ having Fredholm index $\neq 0$ in $\B(e(L^2M))$, then $[T, M]\subset \text{\rm q}\K_M$, 
and the derivation $\delta_e: M \rightarrow \text{\rm q}\K_M$ defined by $\delta_e(x) =[e,x], x\in M$, 
is not inner, i.e., there exists no $K\in \text{\rm q}\K_M$ such that $\delta_e=\text{\rm ad}K$. 
\end{lemma}

\begin{proof} $1^\circ$ By Theorem~\ref{thrm:7.1}, the derivation $\delta=\text{\rm ad}T: M \rightarrow \B(L^2M)$ 
is $\| \cdot \|_2-\text{\rm q}_M$ continuous. Since $M_0$ is $\| \cdot \|_2$-dense in $M$, $[T, M_0]\subset \text{\rm q}\K_M$ 
 and q$\K_M$ is q$_M$-closed in $\B(L^2M)$, it follows that 
$[T, M]\subset \text{\rm q}\K_M$. 

\vskip.05in 

$2^\circ$ Let $A=\{u\}''$.  Since $u$ is a Haar unitary, one can view the restriction of the action of $u$ on $L^2A$ as the bilateral shift on $\ell^2\integers$. 
Denote $u=v\oplus w$ where $v$ is  the restriction of $u$ to $L^2A=\ell^2\integers$ and $w$ its restriction to $L^2M\ominus L^2A$. 
By [BDF73], there exist compact operators $K_0, K_1\in \K(L^2M)$ such that $(u+K_0, e+K_1)$ are unitary conjugate to $(u, f)$, 
where $f$ is the orthogonal projection of $L^2M$ onto $\ell^2\integers_+\subset \ell^2\integers=L^2A$. 

Thus, if  $\delta_e=\text{\rm ad}K$ for some $K\in \K(M, L^1M)$, then $\delta_f = \text{\rm ad}f=\text{\rm ad}(e+K_1)=\text{\rm ad}(K')$  
with $K'=K+K_1 \in \K(M, L^1M)$. This implies $f-K'\in M'=M^{op}$, so there must exist $x_0\in M$ such that 
$f(\hat{y})=K'(\hat{y})+\hat{yx_0}$, for all $y\in M$. Since $\lim_{|n| \rightarrow \infty} K'(\hat{u^n})=0$ and $f(\hat{u^n})$ is equal to $\hat{u^n}$ 
for $n>0$ and is equal to $0$ for $n<0$, this shows on the one hand that $0=\lim_{n\rightarrow \infty}\|f(\hat{u^{-n}})\|_2 = \|u^{-n}x_0\|_2=\|x_0\|_2$, 
on the other hand $1=\lim_{n\rightarrow \infty}\|f(\hat{u^n})\|_2 = \|u^nx_0\|_2= \|x_0\|_2$, a contradiction. 

\end{proof}

\begin{thrm}\label{thrm:7.3}
For any separable diffuse finite von Neumann algebra $M$, there exists a non inner derivations of $M$ into $\text{\rm q}\K_M$.
\end{thrm}
\begin{proof}
Since $M$ is separable, we can fix a weakly dense sequence of $x_n$ in $M$. By \cite{Arv77}, the closed ideal $\K(L^2M)$ of $\B(L^2M)$ has  a quasicentral approximate unit. In particular, for any $\epsilon > 0$ and any operators $T_1, T_1, \dots, T_k \in \B(L^2M)$, we can find an operator $K \in \K(L^2M)_+$ from such a quasicentral approximate unit such that $\|K \| \leq 1$  and  $\|[K, T_i]\| < \epsilon$ for all $1 \leq i \leq k$. Moreover, since such a quasicentral approximate unit weakly tends to the identity, for any $0< \alpha < 1$ such an operator $K$ can be chosen to satisfy $\langle K \hat{1}, \hat{1} \rangle > \alpha$. Thus, we can find a sequence of operators $K_n$ in $\K(L^2M)$ with $\|K_n \| \leq 1$ for all $n\geq 1$ such that $\|[ K_n,x_k]\| < 2^{-n}$ for all $1 \leq k \leq n$ and  $\langle K_n \hat{1}, \hat{1} \rangle > 1/2$ for all $n\geq 1$.

Now fix a sequence of unitaries $u_n$ in $M$ that are weakly tending to 0. We claim there exists a  subsequence $(u_{n_i})_{i=1}^\infty$ such that 
\begin{enumerate}
\item
$\left\|   \sum_{i=1}^n Ju_{n_i}J K_i Ju_{n_i}^* J\right\| <2$ for  all $n \geq 1$;
\item
$|\langle K_i u_{n_j}^*u_{n_i} \hat{1}, u_{n_j}^*u_{n_i} \hat{1} \rangle| < 2^{-i-1}$ for all $i \neq j$.
\end{enumerate}
We construct such a subsequence inductively. First,  let $u_{n_1} = u_1$. Next, assume for some $k \geq 1$ we have found $u_{n_1}, u_{n_2}, \dots, u_{n_k}$ such that the above condition 1 occurs for all $1\leq n \leq k$ and condition 2 occurs for all $1\leq i,j \leq k$ with $i\neq j$. Then notice that for any compact operators $T,S \in K(L^2M)$ and any sequence of unitaries $v_n$ in $M$  converging weakly to 0 we have $\| T + v_n S v_n^*\| \to \max\{\| T\| ,\|S\|\}$ as $n$ tends to infinity. Since $   \sum_{i=1}^k Ju_{n_i}J K_i Ju_{n_i}^* J$  and $K_{k+1}$ are compact operators of norm less than 2, it follows there is an $N_1$ such that  for all $n \geq N_1$ 
$$\left\| Ju_{n}J K_{k+1} Ju_{n}^* J +  \sum_{i=1}^k Ju_{n_i}J K_i Ju_{n_i}^* J\right\| < 2.$$

Next, note that for each fixed $1\leq i \leq k$  we have that $u_nu_{n_i} K_{k+1} u_{n_i}^* u_n^*$ converges weakly to 0 as $n$ tends to infinity. Thus, there is an $N_2$ such that for all $1\leq i \leq k$ and all  $n\geq N_2$ 
$$
|\langle K_{k+1} u_{n_i}^*u_{n} \hat{1}, u_{n_i}^*u_{n} \hat{1} \rangle|
=
|\langle u_nu_{n_i} K_{k+1} u_{n_i}^* u_n\hat{1}, \hat{1}
\rangle|
<
2^{-k-2}.
$$
Similarly, for each fixed $1\leq i \leq k$ we have that $u_n K_i u_n^*$ converges weakly to 0 as $n$ tends to infinity. Hence, there is an $N_3$ such that for all $1\leq i \leq k$ and all $n\geq N_3$ 
$$
|\langle K_{i} u_{n}^*u_{n_i} \hat{1}, u_{n}^*u_{n_i} \hat{1} \rangle|
=
|\langle (u_{n} K_{i} u_{n}^*) u_{n_i}\hat{1}, u_{n_i}\hat{1}
\rangle|
<
2^{-i-1}.
$$
If we take $n_{k+1} = \max\{N_1, N_2, N_3\}$, then the terms $u_{n_1}, u_{n_2}, \dots, u_{n_{k+1}}$ will satisfy the above condition 1 for  all $1\leq n \leq k+1$ and condition 2 occurs for all $1\leq i,j \leq k+1$ with $i\neq j$. By induction, it follows that the desired subsequence $(u_{n_i})_{i=1}^\infty$  exists.

Now we define an operator $T$ by letting
$$T = \sum_{i=1}^\infty Ju_{n_{2i}}J K_{2i} Ju_{n_{2i}}^* J. $$
Note because of how we chose the unitaries $u_{n_i}$ that  $T$ will indeed be a well-defined operator in $\B(L^2M)$ with $\|T\| \leq 2$. Moreover, for any $x_j$ from our weakly dense sequence of $M$ we have 

$$[T, x_j] = 
\sum_{i=1}^\infty [Ju_{n_{2i}}J K_{2i} Ju_{n_{2i}}^*J, x_j]
=
\sum_{i=1}^\infty Ju_{n_{2i}}J [K_{2i},x_j] Ju_{n_{2i}}^* J.$$
Each summand $Ju_{n_{2i}}J [K_{2i},x_j] Ju_{n_{2i}}^* J$ in this series is a compact operator and, because of how we chose the operators $K_n$, for all  $i \geq j/2$ we have 
$$
\| Ju_{n_{2i}}J [K_{2i},x_j] Ju_{n_{2i}}^* J \|
=
\|[K_{2i},x_j] \| 
\leq 
2^{-2i}.
$$
Thus, this is a $\|\cdot\|$-norm convergent series of compact operators, and in turn $[T,x_j]$ is a compact for each $x_j$.  By Lemma~\ref{lemma:7.2}, $\text{\rm ad}T$ is a derivation of $M$ into $\text{\rm q}\K_M$. 

We claim, however, that $\text{\rm ad}T$ is not inner. Otherwise, assume for sake of contradiction there is an $S \in \text{\rm q}\K_M$ such that $\text{\rm ad}S =\text{\rm ad}T$. Take any sequence of unitaries $v_n$ in $M$ that weakly converge to 0. Then we note that since $(T-S)$ commutes with $M$
\begin{equation}
\begin{split}
\langle v_n T v_n^* \hat{1}, \hat{1} \rangle
&=
\langle v_n (T-S) v_n^* \hat{1}, \hat{1} \rangle 
+\langle v_n Sv_n ^* \hat{1}, \hat{1} \rangle\\
&=
\langle (T-S) \hat{1}, \hat{1} \rangle 
+\langle v_n Sv_n ^* \hat{1}, \hat{1} \rangle.\\
\end{split}
\end{equation}
Using lemma~\ref{lemma:5.2}, $\langle v_n Sv_n ^* \hat{1}, \hat{1} \rangle$ must converge to 0 as $n$ tends to infinity. Thus, we observe that $\langle v_n T v_n^* \hat{1}, \hat{1} \rangle$ must converge as $n$ tends to infinity. 

It follows then that $\langle u_{n_j} T u_{n_j}^* \hat{1}, \hat{1} \rangle$ converges as $j$ tends to infinity. However, for even terms of this sequence, we notice that 

\begin{equation*}
\begin{split}
|\langle u_{n_{2j}} T u_{n_{2j}}^* \hat{1},
 \hat{1} \rangle|
 &=
  \left|\sum_{i=1}^\infty\langle u_{n_{2j}}  Ju_{n_{2i}}J K_{2i} Ju_{n_{2i}}^* J u_{n_{2j}}^* \hat{1},
 \hat{1} \rangle \right|
 \\
 &=
  \left|\sum_{i=1}^\infty\langle K_{2i}  u_{n_{2j}}^* u_{n_{2i}}\hat{1},
 u_{n_{2j}}^* u_{n_{2i}}\hat{1} \rangle \right|
 \\
  &\geq
 |\langle K_{2j}  \hat{1},
\hat{1} \rangle |
 -\sum_{1\leq i\neq j} \left|\langle K_{2i}  u_{n_{2j}}^* u_{n_{2i}}\hat{1},
 u_{n_{2j}}^* u_{n_{2i}}\hat{1} \rangle \right|.
 \\
\end{split}
\end{equation*}
Because of how we chose the operators $K_n$, we have $|\langle K_{2j}  \hat{1},
\hat{1} \rangle | > 1/2$, whereas by construction  $\left|\langle K_{2i}  u_{n_{2j}}^* u_{n_{2i}}\hat{1},
 u_{n_{2j}}^* u_{n_{2i}}\hat{1} \rangle \right| < 2^{-2i-1}$ for all $i \neq j$. We then get a lower bound 

$$
|\langle u_{n_{2j}} T u_{n_{2j}}^* \hat{1},
 \hat{1} \rangle|
 >
 1/2
 -\sum_{1 \leq i\neq j} 2^{-2i-1}
 \geq 1/3.
$$
Conversely, for any odd term of this sequence
 \begin{equation}
\begin{split}
|\langle u_{n_{2j+1}} T u_{n_{2j+1}}^* \hat{1},
 \hat{1} \rangle|
 &=
  \left|\sum_{i=1}^\infty\langle u_{n_{2j+1}}  Ju_{n_{2i}}J K_{2i} Ju_{n_{2i}}^* J u_{n_{2j+1}}^* \hat{1},
 \hat{1} \rangle \right|
 \\
 &=
  \left|\sum_{i=1}^\infty\langle K_{2i}  u_{n_{2j+1}}^* u_{n_{2j}}\hat{1},
 u_{n_{2j+1}}^* u_{n_{2i}}\hat{1} \rangle \right|
 \\
  &\leq
\sum_{i=1}^\infty \left|\langle K_{2i}  u_{n_{2j+1}}^* u_{n_{2i}}\hat{1},
 u_{n_{2j+1}}^* u_{n_{2i}}\hat{1} \rangle \right|.
 \\
\end{split}
\end{equation}
Again, using that $\left|\langle K_{2i}  u_{n_{2j+1}}^* u_{n_{2i}}\hat{1},
 u_{n_{2j+1}}^* u_{n_{2i}}\hat{1} \rangle \right|< 2^{-2i-1}$ we get a bound 
$$
|\langle u_{n_{2j+1}} T u_{n_{2j+1}}^* \hat{1},
 \hat{1} \rangle|
 <
\sum_{i=1}^\infty 2^{-2i-1}
= 1/6.
$$
It follows then that the sequence $\langle u_{n_k} T u_{n_k}^* \hat{1}, \hat{1} \rangle$ does not converge. Hence, by contradiction, $\text{\rm ad}T$ must be a non inner derivation.

\end{proof}

\begin{prop}\label{prop:7.4}
 Let $\Gamma$ be a countable group, set $M=L\Gamma$ and let $f\in \ell^\infty\Gamma$ be so that $_gf - f \in c_0(\Gamma)$, $\forall g\in \Gamma$. Denote $T_f\in \B(L^2M)$ 
the diagonal operator corresponding to $f$. 
\vskip.05in

$1^\circ$ We have $[T_f, M] \subset \text{\rm q}\K_M$, and  thus $\delta_f:=\text{\rm ad}T_f$ 
defines a derivation of $M$ into $\text{\rm q}\K_M$. 
\vskip.05in

$2^\circ$ If $f\not\in \mathbb C + c_0(\Gamma)$,  
then the 
derivation $\delta_f$ is outer, i.e,  
there exists no $K\in \text{\rm q}\K_M$ such that $\delta_f=\text{\rm ad}K$. 
\end{prop}

\begin{proof} $1^\circ$ The condition $_gf - f \in c_0(\Gamma)$, $\forall g\in \Gamma$, amounts to $[M_0, T_f]\subset \K(L^2M)\subset \text{\rm q}\K_M$, where $M_0=\complex\Gamma$. 
Since $M_0$ is a weakly  dense $^*$-subalgebra of $M$, by Lemma~\ref{lemma:7.2} it follows that $[M, T_f]\subset \text{\rm q}\K_M$. 

$2^\circ$  
Assume there exists $K\in \text{\rm q}\K_M$ such that $\text{\rm ad}(K)=\text{\rm ad}(T_f)$ on $M$. We let $\E_0: \B(L^2M) \to \ell^\infty \Gamma$ denote the conditional expectation to the diagonal operators given by $\E_0(T)(g) = \langle T \hat{u}_g, \hat{u}_g \rangle$. Notice that $\E_0$ implements the canonical trace on both $L\Gamma$ and $R\Gamma$. By Lemma~\ref{lemma:4.7} if $p \geq 2$, then we have $\| \E_0(T) \| \leq \vertiii{ T }_p$, and so from Theorem~\ref{thrm:6.5} it follows that $\E_0(K) \in c_0(\Gamma)$.

Since $K-T_f\in M'=M^{op}$, we then have $\E_0(K) - f = \E_0( K - T_f) \in \mathbb C$, contradicting the fact that $f \not\in \mathbb C + c_0(\Gamma)$.

\end{proof}

\begin{cor} If $\Gamma$ is any infinite group, then there exists a non-inner derivation of $M = L\Gamma$ into $\text{\rm q}\K_M$ of the form $\delta_f  = \text{\rm ad}T_f$ where $f \in \ell^\infty\Gamma$ is given as in Proposition~\ref{prop:7.4}.

\end{cor}

\begin{proof}
From an argument very similar to the one used in Theorem~\ref{thrm:7.3} it follows that there always exist $f \in \ell^\infty \Gamma$ so that $_gf - f \in c_0(\Gamma)$ for all $g \in \Gamma$, but such that $f \not\in \mathbb C + c_0(\Gamma)$. One simply starts with an asymptotically $\Gamma$-invariant approximate identity in $c(\Gamma)$ and proceeds as in the proof of Theorem~\ref{thrm:7.3}. 

\end{proof}

\begin{lemma}\label{lemma:7.6}  Assume $\delta: M \rightarrow \text{\rm q}\K_M$ is implemented by $T\in \B(L^2M)$. 
If $K_n \in \text{\rm q}\K_M$ are so that $\|K_n \|\leq \|T\|$, $\forall n$, and $\lim_n \text{\rm q}_M([K_n, x], \delta(x))=0$ for all $x$ in some weakly dense 
$^*$-subalgebra $M_0$ of $M$, then this limit holds true for all $x\in M$.

\end{lemma}
 
\begin{proof} Let $y\in (M)_1$. We have to prove that given any $\varepsilon > 0$ there exists $n_0$ such that for any $n\geq n_0$ there exists $p\in \mathcal P(M)$ satisfying 
$\tau(1-p)\leq \varepsilon$ and $\|p(\delta(y)-[K_n, y])p\|\leq \varepsilon$. 

By Kaplanski's theorem, we can take $y_0\in (M_0)_1$ with $\|y_0-y\|_2 \leq (\varepsilon/2)^{3/2}/2$. 
By applying the hypothesis to this $y_0\in M_0$, there exists $n_0$ such that $\forall n\geq n_0$, $\exists p_0\in \mathcal P(M)$ with $\tau(1-p_0)\leq \varepsilon/2$ and 
$\|p_0(\delta(y_0)-[K_n, y_0])p_0\|\leq \varepsilon/3$. On the other hand, by applying Theorem~\ref{thrm:7.1}  to $x=y-y_0$ and the derivations $\delta, \text{\rm ad}(K_n)$, 
we get a projection $p_1\in M$ such that $\tau(1-p_1) \leq \varepsilon/2$ 
and $\|p_1\delta(y-y_0)p_1\| \varepsilon/3$, $\|p_1[K_n, (y-y_0)]p_1\|\leq \varepsilon/3$. Thus, if we let $p=p_0 \wedge p_1$, then $\tau(1-p)\leq \varepsilon$ and for each $n\geq n_0$ we have 
$$
\|p(\delta(y) - [K_n, y])p\| 
$$
$$
\leq \|p\delta(y-y_0)p\|+\|p[K_n, (y-y_0)]p\| + \|p(\delta(y_0)-[K_n, y_0])p\| \leq \varepsilon.
$$

\end{proof}

\begin{thrm}\label{thrm:7.7}
 Let  $\delta: M \to \text{\rm q}\K_M$ be a derivation implemented by $T \in \B(L^2 M)$. Then there exists a net of finite-rank operators $K_\iota$ with $\| K_\iota \| \leq \|T \|$ such that $\lim_\iota \text{ \rm q}_M(\delta(x) , [K_\iota , x] )  = 0$ for all $x \in M$. Moreover, if $L^2M$ is separable, then the net can be taken a sequence. 
\end{thrm}
\begin{proof} 
Let $F = \{x_1, x_2, \dots, x_n\}$ be an arbitrary finite subset of $M$ and $\epsilon>0$. Since $\delta $ is a derivation into $\text{\rm q}\K_M$, we can find a projection $p \in M$ with $\tau(1-p) < \epsilon/2$ such that $p\delta (x_i)p \in \K(L^2M)$   for $1\leq i \leq n$. Consider then the convex subset $C \subset \K(L^2 M)^n$ consisting of all $n$-tuples of the form

$$
(p\delta(x_1)p-p[K, x_1]p , p\delta(x_2)p-p[K, x_2]p,\dots, p\delta(x_n)p-p[K, x_n]p), 
$$
where  $K$ runs over all finite-rank operators in $\B(L^2 M)$ such that $\| K \| \leq \| T\|$.

The set $C \subset \K(L^2 M )^n $ can be viewed as a subset of $(\K(L^2 M)^n)^{**} = \B(L^2 M)^n$. Note that, since $\delta=\text{\rm ad}(T)$, 
if we plug in $T$ for $K$ in the above $n$-tuple viewed as an element in $\B(L^2M)^n$,  then one gets $(0, ..., 0)$.     
Note also that $T$ is a $wo$-limit of finite-rank operators with norm at most $\| T\|$ and that this implies 
$(0, ..., 0)=(p\delta(x_1)p-p[T, x_1]p , p\delta(x_2)p-p[T, x_2]p,\dots, p\delta(x_n)p-p[T, x_n]p)$ is in the 
$\sigma(\B(L^2 M)^n, \B^*_\text{\rm n}(L^2 M)^n)$-closure of $C$ in $\B(L^2 M)^n$. But since $C \subset \K(L^2M )^n$ is convex, its norm closure in $\K(L^2 M)^n$ coincides with its closure in the $\sigma(\K(L^2 M)^n, \B^*_\text{\rm n}(L^2 M)^n)$ topology. Hence $(0, 0, \dots, 0)$ is in the norm closure of $C$.  In particular, there exists a finite-rank operator $K \in \B(L^2 M)$ such that $\| K \| \leq \|T \|$ and $\| p \delta(x_i) p-p[K, x_i]p\|< \epsilon/2$ for all $	1\leq i \leq n$. But then we see that 

$$
\sup_{1\leq i \leq n} 
\text{ \rm q}_M( \delta(x_i), [K, x_i] ) 
\leq 
\sup_{1\leq i \leq n}
\tau(1-p) + \|p(\delta(x_i) - [K,x_i]) p \|
< \epsilon.
$$

This shows that for any set $F \subset M$  we can find a finite-rank operator $K_{F} \in \B(L^2M)$  such that $\| K_{F} \| \leq \|T\|$ and $\text{ \rm q}_M( \delta(x_i) - [K_F, x_i] ) <1/|F|$. This net 
$(K_F)_F$, indexed over all finite subsets will then satisfy the condition.

The fact that this net can be taken to be a sequence when $L^2M$ is separable follows from Lemma~\ref{lemma:7.6}. 

\end{proof}

The next result shows if $B_0 \subset M$ 
is a weakly quasi-regular diffuse von Neumann subalgebra of $M$ (in the sense of \cite{GP14}), 
then the derivations of $M$ into any of the bimodule $\text{\rm q}\K_M$, are uniqueley 
determined by their restriction to $B_0$. 

\begin{prop} 
Let $M$ be a tracial von Neumann algebra with a diffuse weakly quasi-regular von Neumann 
subalgebra $B_0\subset M$.  
If a derivation $\delta: M \rightarrow \text{\rm q}\K_M$ 
vanishes on $B_0$, then $\delta = 0$ on all $M$. 
\end{prop}
\begin{proof} Since $\delta$ is automatically $\| \cdot \|_2$-$\text{\rm q}_M$-continuous, it follows that the space $\tilde{B}$ of elements 
in $M$ on which $\delta$ vanishes (which contains the diffuse algebra $B_0$, by hypothesis)  
is a von Neumann subalgebra of $M$. Let $u$ be a unitary element in $M$ such that $B:=u^*\tilde{B}u\cap \tilde{B}$ is diffuse 
and denote $\sigma: B \rightarrow  M$ the isomorphism of $B$ into $\tilde{B}$ given by $\sigma(b)=ubu^*$, $b\in B$. 
Since $ub=\sigma(b)u$, by applying $\delta$ it follows that $\delta(u)b=\sigma(b)\delta(u)$, $\forall b\in B$. Thus, $K=\delta(u)\in \text{\rm q}\K_M$ 
satisfies the conditions in Lemma~\ref{lemma:6.6}, implying that $\delta(u)=0$. Since $B_0\subset M$ is weakly quasi-regular, 
this shows that $\tilde{B}=M$. 

\end{proof}

Let us end this section by mentioning some $\text{\rm q}_M$-approximation properties of derivations of a tracial von Neumann algebra 
$M$ into Banach $M$-bimodules endowed with the $\text{\rm q}_M$-metric, notably $\text{\rm q}\K_M$.

\begin{prop}
Let $(M,\tau)$ be a tracial von Neumann algebra,  $\B$ a Banach $M$-bimodule and 
$\delta: M \rightarrow \B$ a derivation.  

\vskip.05in 

$1^\circ$  
Let $M_0\subset M$ be a  weakly dense $C^*$-subalgebra and $\B_0\subset \B$ 
an $M$ sub-bimodule $($not necessarily norm-closed$)$. Assume 
$p\delta(M_0)p\subset \B_0$, for some projection $p\in M$. 
Then, for any countable subset $\X\subset M$ and any $\varepsilon_0 > 0$, there exists $p_0\in \mathcal{P}(pMp)$ 
such that $\tau(p-p_0)\leq \varepsilon\tau(p)$ and $p_0\delta(x)p_0 \in \B_0$, $\forall x\in \X$. 

\vskip.05in 

$2^\circ$ 
If $\B=\text{\rm q}\K_M$, then given any  
separable $C^*$-subalgebra $M_0\subset M$ and any $\varepsilon > 0$, there exists $p_0\in \mathcal{P}(M)$ 
such that $\tau(1-p_0)\leq \varepsilon$ and $p_0\delta(x)p_0 \in \K(L^2M)$, $\forall x\in M_0$. 
\end{prop}
\begin{proof}
$1^\circ$  Let $\X=\{x_n\}_{n\geq 1}$ be an enumeration of $\X$. 
By Pedersen's Lusin-type Theorem, for each $n$ 
there exists $p_n\in \mathcal{P}(M)$ and $y_n \in M_0$ such that $x_np_n=y_np_n$, $p_nx_n=p_ny_n$ 
and $\tau(p_n)\geq 1 - \tau(p) \varepsilon/2^{n+1}$, $\forall n$. 
Since $\delta(x_np_n)=y_n\delta(p_n)+\delta(y_n)p_n$, we have 
$$
p_n\delta(x_n)p_n=p_n\delta(x_np_n)p_n- p_nx_n\delta(p_n)p_n
$$
$$
= p_ny_n\delta(p_n)p_n+p_n\delta(y_n)p_n-p_nx_n\delta(p_n)p_n=p_n\delta(y_n)p_n \in \B_0. 
$$

Thus, if we let $p_0=\wedge_{n\geq 1} p_n \wedge p$, then $p_0\delta(x)p_0 \in \B_0$, $\forall x\in \X$. 
Moreover, we have $\tau(\wedge_{n\geq 1} p_n)\geq (1-\Sigma_{n\geq 1} \tau(1-p_n)) = 1-\varepsilon \tau(p)$
and thus $\tau(p_0) \geq (1-\varepsilon\tau(p))+\tau(p)-1=(1-\varepsilon)\tau(p)$, 
implying that $\tau(p-p_0)\leq \varepsilon \tau(p)$.  

\vskip.05in

$2^\circ$ This is trivial by Proposition~\ref{prop:6.2}. 

\end{proof}


\begin{thebibliography}{9}

\bibitem[A14]{A14} V. Alekseev, {\em On the first continuous $L^2$-cohomology for free group factors}, J. Funct. Anal. {\bf 267} (2014), no. 11, 4264-4279.


\bibitem[AK15]{AK15} V. Alekseev,  D. Kyed, {\em Measure continuous derivations on von Neumann algebras and applications to $L^2$-cohomology}, {\em J. Operator Theory}, {\bf 73} (2015), no. 1, 91-111.

\bibitem[Arv77]{Arv77}    W. Arveson. {\em Notes on extensions of C$^*$-algebras}. Duke Math. J., {\bf 44} (1977), no. 2, 329-355.

\bibitem[BDF73]{BDF73} L.G. Brown, R.G. Douglas, P.A. Fillmore: ``Unitary equivalence modulo the compact operators and extensions 
of $C^*$-algebras'', Lecture Notes in Math. No. {\bf 345}, Springer-Verlag, 1973. 


\bibitem[Ch80]{Ch80} E. Christensen: {\em Extensions of derivations} II. Mathematica Scandinavica {\bf 50} (1982), 111-122. 


\bibitem[ChPSS03]{ChPSS03} E. Christensen, F. Pop, A. Sinclair, R. Smith: 
{\em Property Gamma factors and the Hochschild cohomology problem},  Proc. Natl. Acad. Sci. USA {\bf 100} (2003), 3865-3869. 

\bibitem[C75]{C75}
A.~Connes, \emph{On the classification of von {N}eumann algebras and their
  automorphisms}, Symposia {M}athematica, {V}ol. {XX} ({C}onvegno sulle
  {A}lgebre {$C^*$} e loro {A}pplicazioni in {F}isica {T}eorica, {C}onvegno
  sulla {T}eoria degli {O}peratori {I}ndice e {T}eoria {$K$}, {INDAM}, {R}ome,
  1975), Academic Press, London, 1976, pp.~435--478.



\bibitem[C76]{C76} A. Connes, {\em Classification of injective factors}, Ann. of Math., {\bf 104} (1976), 73-115. 

\bibitem[C01]{C01} A. Connes, {\rm Factors and geometry}, lecture at MSRI, May 1 2001.  

\bibitem[CS05]{CS05} A. Connes, D. Shlyakhtenko, {\em $L^2$-homology for von Neumann algebras}, J. Reine Angew. Math. {\bf 586} (2005) 125--168

\bibitem[DKEP22]{DKEP22} C. Ding, S. Kunnawalkam Elayavalli, J. Peterson, {\em Properly proximal von Neumann algebras}, preprint, arXiv:2204.00517, 2022, to appear in Duke Math. J. 

\bibitem[E88]{E88} E. Effros, {\em Amenability and virtual diagonals for von Neumann algebras}, J. Funct. Analysis {\bf 78} (1988), 137-153.

\bibitem[G01]{G01} D. Gaboriau: {\em Invariants $\ell^2$ de relations d’equivalences et de groupes}, Publ. Math IHES {\bf 95} (2002), 93-150.

\bibitem[GP14]{GP14} A. Galatan, S. Popa, {\em Smooth bimodules and cohomology of} II$_1$ {\em factors}, Journal of Inst. Math. Jussieu  {\bf 16} (2017), 155-187 


\bibitem[H45]{H45} G. Hochschild: {\em On the cohomology groups of an associative algebra}, Annals of Mathematics, {\bf 46} (1945),  58-67. 

\bibitem[Ho77]{Ho77} T. Hoover: {\em Derivations, homomorphisms, and operator ideals}. Proc Amer. Math Soc, {\bf 62} (1977), 293-298. 

\bibitem[J72]{J72} B. Johnson: ``Cohomology of Banach algebras'', Memoirs of the AMS, {\bf 127}, 1972. 

\bibitem[J74]{J74} B. Johnson: {\em A class of} II$_1$ {\em  factors without property P but with zero second cohomology}, Ark. Mat. {\bf 12} (1974), 153-159. 

\bibitem[JKR72]{JKR72} B. Johnson, R.V. Kadison, J. Ringrose: {\em Cohomology of operator algebras} III: {\em Reduction to normal cohomology}, Bull. Soc. Math. France {\bf 100} (1972), 73-96.

\bibitem[JP72]{JP72} B. Johnson, S. Parrott, {\em Operators commuting with a von Neumann algebra modulo the set of compact operators}, J. Funct. Analysis {\bf 11}  (1972),  39-61.

\bibitem[K66]{K66} R. Kadison: {\em Derivations of operator algebras}, Ann Math, {\bf 83} (1966), 280-293.

\bibitem[KR71]{KR71} R. V. Kadison,  J. R. Ringrose, {\em Cohomology of operator algebras I: Type I von Neumann algebras},  Acta Math.,  {\bf 126} (1971), 227-243. 

\bibitem[K77]{K77} S. Kaijser {\em A note on dual Banach spaces},  Mathematica Scandinavica,,  {\bf 41} (1977), 325-330. 

\bibitem[Ka53]{Ka53} I Kaplansky: {\it Modules over operator algebras}, Amer. J. Math, {\bf 75} (1953) 839-858. 

\bibitem[Ma00]{Ma00} B. Magajna: {\it $C^*$-convex sets and completely bounded bimodule homomorphisms}, Proc. of the Royal Soc. of Edinburgh, {\bf 130A} (2000), 275-387.

\bibitem[O10]{O10} N. Ozawa, {\em A comment on free group factors}, 
Banach Center Publ., {\bf 89} (2010), 241-245.  

\bibitem[Pe09]{Pe09} J. Peterson, {\em $L^2$
-rigidity in von Neumann algebras}, 
 Invent. Math., {\bf 175} (2009), 417–433.  

\bibitem[Pi01]{Pi01} G. Pisier: ``Similarity problems and completely bounded maps'', 
 Notes in Mathematics, {\bf 1618}, Springer-Verlag, Berlin, 2001. 

\bibitem[P85]{P85} S. Popa, {\em The commutant modulo the set of compact
operators of a von Neumann algebra}, J. Funct. Analysis, {\bf 71} (1987), 393-408.

\bibitem[P01]{P01} S. Popa: {\em On a class of type} II$_1$ {\em factors with
Betti numbers invariants}, Ann. of Math {\bf 163} (2006), 809-899
(math.OA/0209310; MSRI preprint 2001-024, May 2001).

\bibitem[PR89]{PR89} S. Popa, F. Radulescu, {\em Derivations of von Neumann
factors into the compact ideal space of a semifinite algebra}, Duke Math. J., {\bf 57} (1988), 485-518.

\bibitem[PV15]{PV15} S. Popa, S. Vaes, {\em Vanishing of the continuous first $L^2$-cohomology for} II$_1$  {\it factors}, Int. Math. Res. Not. 2015, no. 12, 3899-3907.

\bibitem[R72]{R72}  J. Ringrose: {\em Automatic continuity of derivations of operator algebras}, J. London Math. Soc. {\bf 5} (1972), 432-438. 

\bibitem[Sa66]{Sa66} S. Sakai: {\em Derivations of} W$^*$-{\em algebras}, Ann Math, {\bf 83} (1966), 273-279.

\bibitem[SW55]{SW55} I. M. Singer, J. Wermer: {\em Derivations on commutative normed algebras}, 
Mathematische Annalen {\bf 129} (1955),  260-264. 

\bibitem[Th08]{Th08} A. Thom: {\em $L^2$-cohomology for von Neumann algebras}. Geom. Funct. Anal. {\bf 18} (2008), 251-270. 
\end{thebibliography}
\end{document}